\documentclass[11pt,reqno,oneside]{amsart}
\usepackage{graphicx}
\usepackage{amssymb}
\usepackage{epstopdf}
\usepackage{amssymb}
\usepackage[a4paper, total={6.1in, 9.3in}]{geometry}

\usepackage{bm}
\usepackage{amsmath,amsfonts,amsthm,mathrsfs,amssymb,cite}
\usepackage[usenames]{color}

\usepackage{subfigure}
\usepackage{framed}

\usepackage{enumerate}
\usepackage{hyperref}
\usepackage{xcolor}
\hypersetup{
	colorlinks,
	linkcolor={blue!80!black},
	urlcolor={blue!80!black},
	citecolor={blue!80!black}
}
\usepackage[capitalize,noabbrev]{cleveref}

\numberwithin{equation}{section}

\theoremstyle{definition}

\newtheorem{thm}{Theorem}[section]

\newtheorem{lem}{Lemma}[section]
\newtheorem{prop}{Proposition}[section]
\theoremstyle{definition}
\newtheorem{defn}{Definition}[section]
\theoremstyle{remark}

\newtheorem{rem}{Remark}[section]

\numberwithin{equation}{section}

\newcounter{saveeqn}

\def\bsi{{\mathrm{i}}}

\newcommand{\ba}{\begin{array}}
	\newcommand{\ea}{\end{array}}

\newcommand{\bea}{\begin{eqnarray*}}
	\newcommand{\eea}{\end{eqnarray*}}
\newcommand{\bean}{\begin{eqnarray}}
	\newcommand{\eean}{\end{eqnarray}}

\newcommand{\Blu}[1]{{\color{black}#1}}

\newcommand{\bu}{\mathbf{u}}
\newcommand{\bx}{\mathbf{x}}
\newcommand{\by}{\mathbf{y}}
\newcommand{\bGa}{\mathbf{\Gamma}}
\newcommand{\bvarphi}{\bm{\varphi}}

\newcommand{\bmf}[1]{{\mathbf{#1}}}

\allowdisplaybreaks

\title[Spectral properties of an acoustic-elastic transmission eigenvalue problem]{Spectral properties of an acoustic-elastic transmission eigenvalue problem with applications}

\author{Huaian Diao}
\address{School of Mathematics, Jilin University, Changchun 130012, China}
\email{hadiao@gmail.com, diao@jlu.edu.cn}

\author{Hongjie Li}
\address{Department of Mathematics, The Chinese University of Hong Kong, Hong Kong SAR, China}
\email{hongjieli@cuhk.edu.hk}

\author{Hongyu Liu}
\address{Department of Mathematics, City University of Hong Kong, Kowloon, Hong Kong SAR, China}
\email{hongyu.liuip@gmail.com, hongyliu@cityu.edu.hk}

\author{Jiexin Tang}
\address{School of Mathematics, Northeast Normal University, Changchun 130024, China}
\email{jiexintang@foxmail.com}

\date{} 
\begin{document}
	
	\begin{abstract}
		
		We are concerned with a coupled-physics spectral problem arising in the coupled propagation of acoustic and elastic waves, which is referred to as the acoustic-elastic transmission eigenvalue problem. There are two major contributions in this work which are new to the literature. First, \Blu{under a mild condition on the medium parameters, we prove the existence of an acoustic-elastic transmission eigenvalue.} Second, we establish a geometric rigidity result of the transmission eigenfunctions by showing that they tend to localize on the boundary of the underlying domain. Moreover, we also consider interesting implications of the obtained results to the effective construction of metamaterials by using bubbly elastic structures and to the inverse problem associated with the fluid-structure interaction.

		\medskip


		
	\end{abstract}
	
	\maketitle
	
	\section{Introduction}
	
	\subsection{Mathematical setup and summary of major findings}
	
	Initially focusing on the mathematics, but not the physics, we present the mathematical setup of our study.  Let $ \Omega $ be a  simply connected domain with $\partial\Omega\in C^{1, s}$ for some $0<s<1$ in $ \mathbb{R}^N, N=2,3 $. The complement $  \mathbb{R}^N\backslash\overline\Omega$ is connected. Henceforth, we let $\nu\in\mathbb{S}^{N-1}$ signify the exterior unit normal vector to $\partial\Omega$. Let $ \rho_b, \rho_e \in \mathbb{R}_{+} $ and $ \kappa\in \mathbb{R}_{+}$. Let $\tilde{\lambda}, \tilde{\mu}$ be the Lam\'e  constants
	which satisfy the following strong convexity conditions
	\begin{equation}\label{eq:conv} 
		\tilde{\mu}> 0 \ \ \mbox{and} \ \ N\tilde{\lambda}+2\tilde{\mu}> 0.
	\end{equation}
	Let $\mathbf{u}(\mathbf{x})$, $\mathbf{x}\in\Omega$, be a $\mathbb{C}^N$-valued function.

	Define the Lam\'e operator $\mathcal{L}_{\tilde{\lambda},\tilde{\mu}}$ and the traction operator $T_{\tilde{\nu}}$ respectively as follows:
	\begin{equation}\label{eq:lo1}
		\begin{split}
			\mathcal{L}_{{\tilde{\lambda}}, {\tilde{\mu}}}\mathbf{u}(\mathbf{x}): =& \tilde{\lambda} \Delta \mathbf{u}(\mathbf{x})+(\tilde{\lambda}+\tilde{\mu})\nabla (\nabla  \cdot \mathbf{u}(\mathbf{x}) ),\quad \mathbf{x}\in\Omega, \\
			T_{\tilde\nu} \bmf {u}(\mathbf{x}):=& \tilde{\lambda}(\nabla \cdot \bmf{u}(\mathbf{x}))\nu(\mathbf{x})+2\tilde{\mu} (\nabla^s\bmf{u}(\mathbf{x}))\nu(\mathbf{x}),\quad\mathbf{x}\in\partial\Omega,
		\end{split}
	\end{equation}
	where
	\begin{equation} \label{div u}
		\nabla^s\bmf{u}:=\frac{1}{2}(\nabla\bmf{u}+\nabla{\bmf{u}^\top }). 
	\end{equation}
	Let $\omega\in\mathbb{R}_+$ and $v(\mathbf{x})$, $\mathbf{x}\in\Omega$, be a $\mathbb{C}$-valued function. We are concerned with the following spectral problem for $(\mathbf{u}, v)\in H^1(\Omega)^N\times H^1(\Omega)$: 
	\begin{equation}\label{eq:syst2}
		\begin{cases}
			\mathcal{L}_{{\tilde{\lambda}}, {\tilde{\mu}}}\bmf {u}(\mathbf{x}) + \omega^2\rho_e\bmf {u}(\mathbf{x})= 0 &\mbox{in}\quad \Omega, \smallskip\\
			\displaystyle{\nabla\cdot(\frac{1}{\rho_b}\nabla v(\mathbf{x}))+\frac{\omega^2}{\kappa}v(\mathbf{x})=0 }&  \mbox{in}\quad \Omega ,\smallskip\\
			\displaystyle{\bmf {u}(\mathbf{x})\cdot \nu-\frac{1}{\rho_b\omega^2}\nabla v(\mathbf{x})\cdot \nu =0} & \mbox{on}\quad \partial \Omega,\smallskip\\
			T_{\tilde{\nu}} \bmf {u}+v(\mathbf{x})\nu =0 & \mbox{on}\quad \partial \Omega. 
		\end{cases}
	\end{equation}	
	In the physical setup, the first equation in \eqref{eq:syst2} is known as the Lam\'e equation which describes the propagation of elastic deformation, whereas the second one is the Helmholtz equation which governs the acoustic wave propagation. Hence, \eqref{eq:syst2} is a coupled-physics spectral problem and shall be referred to as the acoustic-elastic transmission eigenvalue problem in what follows.  In \eqref{eq:syst2}, $ \kappa $ and $\rho_b$ describe the bulk modulus and density  of the acoustic medium. The physical parameters $ \tilde{\lambda} $ and $ \tilde{\mu} $  characterize  the compressional modulus and the shear modulus of the elastic material, respectively. Furthermore, $\rho_e$ specifies the density of the elastic medium.

	It is clear that $(\mathbf{u}, v)\equiv (\mathbf{0}, 0)$ is a pair of trivial solutions to \eqref{eq:syst2}. If there exists a nontrivial pair of solutions to \eqref{eq:syst2}, $\omega\in\mathbb{R}_+$ is called an AE (acoustic-elastic) transmission eigenvalue and $(\mathbf{u}, v)$ is the associated pair of transmission eigenfunctions. It is particularly noted that if $v\equiv 0$, one has from \eqref{eq:syst2} that
	\begin{equation}\label{eq:syst2n1}
		\begin{cases}
			\mathcal{L}_{{\tilde{\lambda}}, {\tilde{\mu}}}\bmf {u}(\mathbf{x}) + \omega^2\rho_e\bmf {u}(\mathbf{x})= 0 &\mbox{in}\quad \Omega, \smallskip\\
			\displaystyle{\bmf {u}(\mathbf{x})\cdot \nu=0,\quad T_{\tilde{\nu}} \bmf {u}=0} & \mbox{on}\quad \partial \Omega,
		\end{cases}
	\end{equation}		
	which is the classical Jones eigenvalue problem. The Jones eigenvalue problem arises in studying the fluid-structure interaction \cite{J} and has been extensively studied in the literature \cite{DNS, H, NSS, LM}. It is known that under certain conditions of the medium parameters $\tilde\lambda, \tilde\mu, \rho_e$ as well as the domain $\Omega$, there exist Jones eigenvalues \cite{DNS}. Clearly, Jones eigenvalues to \eqref{eq:syst2} are a special subset of the AE transmission eigenvalues to \eqref{eq:syst2n1}. In this paper, we show that in addition to the Jones eigenvalues, there exist AE transmission eigenvalues to \eqref{eq:syst2} under two generic scenarios. In the first scenario with no geometric restriction on $\Omega$ but a minor condition on the medium parameters, {we show the existence of AE transmission eigenvalues.} We employ the layer potential theory and Gohberg-Sigal theory to establish the aforementioned result. In the second scenario, if $\Omega$ is of a radial shape and no restriction is imposed on the medium parameters, we show the existence of infinitely many AE transmission eigenvalues. The derivation of the second result employs Fourier series expansions and involves highly technical and subtle calculations. The results are contained in Sections~\ref{sect:2} and \ref{sect:3} in what follows. 
	
	In addition to the spectral properties of the transmission eigenvalues, we further show that the transmission eigenfunctions tend to localize on $\partial\Omega$ in the sense that their $L^2$-energies tend to concentrate on $\partial\Omega$. In fact, we rigorously show the existence of a sequence of transmission eigenfunctions $(\mathbf{u}_m, v_m)_{m\in\mathbb{N}}$ associated with $\omega_m\rightarrow\infty$ such that the aforementioned boundary localization pattern occurs. Furthermore, it is intriguing to note that depending on the configuration of the medium parameters, it may happen that both $\mathbf{u}_m$ and $v_m$ are boundary-localized, which are referred to as a pair of bi-localized eigen-modes; and it may also happen that only $v_m$ is boundary-localized whereas $\mathbf{u}_m$ is not, which are referred to as a pair of mono-localized eigen-modes. The results of this part are rigorously proved in the radial case and numerically verified in the non-radial case, which are contained in Section~\ref{sect:3}.

	\subsection{Physical relevance and implications}
	
	We discuss the physical background and motivation of our study and the interesting implications that it may have. We first consider the fluid-structure interaction problem which is described by the following PDE system (cf. \cite{J,DNS,LM}):
	\begin{equation}\label{eq:fs1}
		\begin{cases}
			\mathcal{L}_{{\tilde{\lambda}}, {\tilde{\mu}}}\bmf {u}(\mathbf{x}) + \omega^2\rho_e\bmf {u}(\mathbf{x})= 0 &\mbox{in}\quad \Omega,\smallskip\\
			\displaystyle{\nabla\cdot(\frac{1}{\rho_b}\nabla v^t(\mathbf{x}))+\frac{\omega^2}{\kappa}v^t(\mathbf{x})=0 }&  \mbox{in}\quad \mathbb{R}^N\backslash\overline\Omega ,\smallskip\\
			\displaystyle{\bmf {u}(\mathbf{x})\cdot \nu=\frac{1}{\rho_b\omega^2}\nabla v^t(\mathbf{x})\cdot \nu, \ T_{\tilde{\nu}} \bmf {u}=-v^t(\mathbf{x})\nu} & \mbox{on}\quad \partial \Omega,\smallskip\\
			v^t-v^i\ \ \mbox{satisfies the Sommerfeld radiation condition}, 
		\end{cases}
	\end{equation}
	where $v^i$ is an entire solution to $(\Delta+\mathfrak{m}^2) v^i=0$ in $\mathbb{R}^N$ with $\mathfrak{m}:=\omega/c_b$ and $c_b:=\sqrt{\kappa/\rho_b}$. The last condition in \eqref{eq:fs1} signifies that the scattered field $v^s:=v^t-v^i$ satisfies 
	\begin{equation}\label{eq:radi1}
		\lim_{r\rightarrow\infty} r^{(N-1)/2}(\partial_r v^s-\mathrm{i}\mathfrak{m} v^s)=0, \ \ \ r:=|\mathbf{x}|. 
	\end{equation}
	The well-posedness of the forward problem \eqref{eq:fs1} is known. In the physical setup, $(\Omega; \tilde\lambda, \tilde\mu, \rho_e)$ signifies an elastic body which is embedded in a fluid whose acoustic property is characterized by $\kappa$ and $\rho_b$. $v^i$ denotes an incident acoustic wave with $\omega\in\mathbb{R}_+$ being its angular frequency and its impingement on the elastic body generates the scattering phenomenon where the elastic response inside the elastic body is governed by the first equation in \eqref{eq:fs1} and the acoustic wave propagation outside $\Omega$ is governed by the second equation in \eqref{eq:fs1}. The elastic and acoustic fields are coupled together via the transmission conditions across $\partial\Omega$. An inverse problem of practical importance in sonar technology is to identify the solid structure $\Omega$ by knowledge of the acoustic field $v^s$ away from $\Omega$. It is clear that if $v^s\equiv 0$, i.e. $v^t=v^i$ in $\mathbb{R}^N\backslash\overline{\Omega}$, the solid body is invisible with respect to the acoustic scanning by using $v^i$. In such a case, one can directly verify that $v=v^i|_{\Omega}$ and $\mathbf{u}$ fulfils \eqref{eq:syst2}, namely they form a pair of AE transmission eigenfunctions. That is, if invisibility occurs, the scattering pattern, namely the perturbative wave pattern, is trapped inside the solid body. Hence, in order to understand the invisibility phenomenon, one needs to study the AE transmission eigenvalue problem \eqref{eq:syst2}. In fact, by following a similar argument in \cite{BL1} for the acoustic transmission eigenvalue problem, one can show that if $(v, \mathbf{u})$ is a pair of AE transmission eigenfunctions, then $v$ can be extended by the Herglotz approximation to form an incident field whose impingement on $\Omega$ generates a nearly-vanishing scattered field. Hence, our results on the existence of AE transmission eigenvalues indicates that (near) invisibility is not a sporadic phenomenon. Moreover, our results on the boundary localization properties of the AE transmission eigenfunctions characterize the quantitative behaviours of both the elastic and acoustic fields when (near) invisibility occurs. It is also interesting to note the mono-localized eigen-modes (i.e. $v$ is boundary-localized but $\mathbf{u}$ is not) can be used for the design of one-way information transmission in that the acoustic observable is void outside the solid, but the elastic observable inside the solid is non-void. Finally, we would like to mention that the spectral pattern of boundary localization was recently investigated in \cite{CDHLW21,DJLZ} for the acoustic transmission eigenfunctions and it has been used to produce a super-resolution scheme for acoustic imaging. By following a similar spirit, one can also use the spectral pattern discovered in the current article to develop novel imaging scheme for the fluid-structure interaction problem. We shall consider these and other developments in a forthcoming paper. 
	
	Next, we consider the elastodynamics in bubbly elastic media. For simplicity, we consider the case with a single air bubble $(\Omega; \kappa, \rho_e)$ embedded in an elastic medium $(\mathbb{R}^N\backslash\overline{\Omega}; \tilde\lambda, \tilde\mu, \rho_b)$. The linear elastic deformation is governed by the following PDE system (cf. \cite{LLZ}):
	\begin{equation}\label{eq:fs2}
		\begin{cases}
			\mathcal{L}_{{\tilde{\lambda}}, {\tilde{\mu}}}\bmf {u}^t(\mathbf{x}) + \omega^2\rho_e\bmf {u}^t(\mathbf{x})= 0 &\mbox{in}\quad \mathbb{R}^N\backslash\overline\Omega,\smallskip\\
			\displaystyle{\nabla\cdot(\frac{1}{\rho_b}\nabla v(\mathbf{x}))+\frac{\omega^2}{\kappa}v(\mathbf{x})=0 }&  \mbox{in}\quad \Omega ,\smallskip\\
			\displaystyle{\bmf {u}^t(\mathbf{x})\cdot \nu=\frac{1}{\rho_b\omega^2}\nabla v(\mathbf{x})\cdot \nu, \ T_{\tilde{\nu}} \bmf {u}^t=-v(\mathbf{x})\nu} & \mbox{on}\quad \partial \Omega,\smallskip\\
			\mathbf{u}^t-\mathbf{u}^i\ \ \mbox{satisfies the Kupradze radiation condition}, 
		\end{cases}
	\end{equation}
	where $\mathbf{u}^i$ is an entire solution to $\mathcal{L}_{{\tilde{\lambda}}, {\tilde{\mu}}}\bmf {u}^i + \omega^2\rho_e\bmf {u}^i= 0$ in $\mathbb{R}^N$. The Kupradze radiation condition is given by decomposing the elastic field $\mathbf{u}^s:=\mathbf{u}^t-\mathbf{u}^i$ into its shear and compressional parts and requiring that each part fulfils the Sommerfeld radiation condition \eqref{eq:radi1}; see e.g. \cite{L,LLL} for a more detailed description. Clearly, one can consider the invisibility issue for \eqref{eq:fs2} and the identically vanishing of $\mathbf{u}^s$ leads again to the transmission eigenvalue problem \eqref{eq:syst2} with $\mathbf{u}=\mathbf{u}^i|_{\Omega}$. However, we are more interested in understanding under what conditions such that when $\mathbf{u}^i\equiv 0$, there exists a nontrivial solution to \eqref{eq:fs2}. In fact, a systematic study was provided in \cite{LLZ} and it is shown that if $\rho_e/\rho_b\gg 1$ and $\omega\ll 1$, the aforementioned resonance phenomenon indeed may happen (at least asymptotically). This low-frequency resonance phenomenon is referred to as the Minnaert resonance and forms the fundamental basis for the effective realisation of elastic metamaterials by bubble-elastic structures (cf. \cite{JBV,LLZ,SMG}). Intriguingly, the Minnaert resonant mode possesses the same boundary-localization pattern as the AE transmission eigenfunctions (cf. \cite{DLT}), which defines the polarizability of the nano-bubbles. Hence, we expect that the discovery in the current article may produce interesting applications in effective construction of elastic metamaterials, which is definitely worth our further study. 
	
	\subsection{Comments on our study} 
	
	\Blu{The AE transmission eigenvalue problem \eqref{eq:syst2} was first considered in \cite{KR}, and it is proved that if the AE transmission eigenvalue exist, they form a discrete set and can accumulate only at $\infty$. }To our best knowledge, we proved the first general result in the literature that under generic scenarios, {there indeed exist AE transmission eigenvalues with a mil dcondition on the medium parameters}. The boundary-localization of transmission eigenfunctions was first discovered in \cite{CDHLW21} for the acoustic transmission eigenvalue problem. It is further extended to the Maxwell system for electromagnetic transmission eigenfunctions in \cite{DLWW}, and to the Lam\'e system for elastic transmission eigenfunctions \cite{JLZZ}. It seems that the boundary-localization phenomenon seems to be universal for different waves when invisibility/non-scattering occurs. However, we would like to emphasize that the different physics underlying this intriguing phenomenon leads to different technical challenges and mathematical treatments, as well as different applications. Indeed, we note that the AE transmission eigen-system \eqref{eq:syst2} consists of two PDEs with one vector-valued and the other one scalar-valued, and moreover the eigenvalue $\omega$ is also coupled in the boundary transmission conditions, which give rise to significant technical difficulties compared to the studies in \cite{CDHLW21, DLWW}. The boundary-localization properties in \cite{CDHLW21} lead to a super-resolution imaging scheme and generating the so-called pseudo plasmon modes, and in \cite{DLWW} lead to an artificial mirage scheme. As discussed earlier, the boundary-localization properties in the current article may have the potential to be applied to the fluid-structure inverse problems and the effective construction of elastic metamaterials. Finally, we would like to mention that in all of the aforementioned works \cite{CDHLW21,DLWW,JLZZ}, the boundary-localization properties were all rigorously verified for the radial geometry and numerically verified for the general geometry. The only exception is in \cite{CDLS} where the general geometry is considered and the boudnary-localization is rigorously justified for the acoustic transmission eigenfunctions via the theory of pseudo-differential operators and generalised Wely's law. However, due to the technical requirement, \cite{CDLS} actually studies the so-called generalised transmission eigenfunctions and moreover the boundary-localization properties are not so sharp and thorough compared to the relevant results for the radial geometry. Hence, we shall follow a similar spirt in the current study by focusing on treating the boundary-localization for the radial geometry rigorously and the general geometry numerically. We shall consider justifying the general geometry in a forthcoming paper.

	\section{Existence of infinitely many transmission eigenvalues}\label{sect:2}

	In this section, we show that in a certain generic scenario {there exist transmission eigenvalues for the system \eqref{eq:syst2}. }
	
	\subsection{Parameter configuration and nondimensionalization}\label{sub:pa}
	
	To facilitate analyzing the coupled-physics system \eqref{eq:syst2}, we introduce the following nondimensional parameters. Define 
	\begin{equation}\label{eq:node}
		\delta=\rho_b/\rho_e, \quad \tau=\frac{c_b}{\tilde{c_p}}=\frac{\sqrt{\kappa/\rho_b}}{\sqrt{(\tilde{\lambda}+2\tilde{\mu})/\rho_e}}, \quad c_b=\sqrt{\kappa/\rho_b}, \quad \tilde{c_p}=\sqrt{(\tilde{\lambda}+2\tilde{\mu})/\rho_e}. 
	\end{equation}
	Let $ l_{\Omega } $ be the average length of the domain $ \Omega $ and we denote the following parameters by
	\begin{equation}\begin{split}\label{eq:para}
			{\mathbf{x}}^{\prime}=\frac{\mathbf{x}}{l_{\Omega }},\ \ k&= \frac{\omega}{c_b} l_{\Omega },\ \ {\mathbf{u}}^{\prime}=\frac{\mathbf{u}}{l_{\Omega }},\\
			\mu=\frac{\tilde{\mu}}{\tilde{\lambda}+2\tilde{\mu}},\ \  \lambda&=\frac{\tilde{\lambda}}{\tilde{\lambda}+2\tilde{\mu}},\ \ {v}^{\prime}=\frac{v}{\rho_bc_b^2}.
	\end{split}\end{equation}
	We would like to mention that the parameters defined in \eqref{eq:node} and \eqref{eq:para} are all nondimensional.
	Through substituting these parameters into \eqref{eq:syst2} and dropping their primes, we obtain the following nondimensional coupled PDE system (cf. \cite{LLZ}):
	\begin{equation}\label{eq:syst3}
		\begin{cases}
			\mathcal{L}_{\lambda, \mu}\bmf {u}(\mathbf{x}) + k^2\tau^2\bmf {u}(\mathbf{x})= 0 &\mbox{in}\quad \Omega,\smallskip \\
			\Delta v(\mathbf{x})+k^2v(\mathbf{x})=0 &  \mbox{in}\quad \Omega ,\smallskip\\
			\bmf {u}(\mathbf{x})\cdot \nu-\frac{1}{k^2}\nabla v(\mathbf{x})\cdot \nu =0 & \mbox{on}\quad \partial \Omega,\smallskip\\
			T_{\nu} \bmf {u}+\delta\tau^2v(\mathbf{x})\nu =0 & \mbox{on}\quad \partial \Omega.
		\end{cases}
	\end{equation}
	It is remarked that the system \eqref{eq:syst3} is equivalent to the original system \eqref{eq:syst2}. Indeed, one can obtain the system \eqref{eq:syst2} from \eqref{eq:syst3} by substituting the parameters in  \eqref{eq:node} and \eqref{eq:para} into the system \eqref{eq:syst3}. Consequently, in what follows we focus on studying the system \eqref{eq:syst3} instead of the system \eqref{eq:syst2}.

	\subsection{Layer potentials and integral reformulation}
	
	We shall rely on the layer potential theory to reformulate the transmission eigenvalue problem \eqref{eq:syst3} into an eigenvalue problem associated with a system of integral equations. To that end, we first introduce the layer potential operators for our subsequent use. 
	
	Let $G^k(\bx)$ be the fundamental solution of the operator $\triangle+k^2$(cf. \cite{LL1}), namely
	\begin{equation}\label{eq:fu_he}
		G^{k}(\bx)=
		\left\{
		\begin{array}{ll}
			\displaystyle{ -\frac{\mathrm{i}}{4} H_0^{(1)}(k|\bx|)}, & N=2 ,\medskip \\
			\displaystyle{-\frac{e^{\mathrm{i} k|\bx|}}{4\pi |\bx|}}, & N=3 ,
		\end{array}
		\right.
	\end{equation}
	where $H_0^{(1)}$ is the zeroth-order Hankel function of the first kind. The single layer potential associated with the Helmholtz system is defined by
	\begin{equation}\label{eq:s_h}
		S_{\partial \Omega}^{k}[\varphi](\bx)=\int_{\partial \Omega}G^k(\bx-\by)\varphi(\by)ds(\by) \quad \bx\in\mathbb{R}^N, 
	\end{equation}
	with $\varphi(\bx)\in H^{-1/2}(\partial \Omega)$. Then the conormal derivative of the single layer potential enjoys the jump formula
	\begin{equation}\label{eq:ju_he}
		\nabla S_{\partial \Omega}^{k}[\varphi] \cdot \nu|_{\pm}(\bx)=\left(\pm\frac{1}{2}I +K_{\partial \Omega}^{k,*} \right)[\varphi](\bx) \quad \bx\in\partial \Omega,
	\end{equation}
	where $I$ is an identity operator and
	\[
	K_{\partial \Omega}^{k,*}[\varphi](\bx)= \int_{\partial \Omega} \nabla_{\bx} G^k(\bx-\by)\cdot \nu_{\bx}\varphi(\by)ds(\by) \quad \bx\in\partial \Omega,
	\]
	which is also known as the Neumann-Poincar\'e (N-P) operator associated with Helmholtz system. Here and also in what follows, the subscript $\pm$ indicates the limits from outside and inside $\Omega$, respectively. 
	
	For the Lam\'e system, the fundamental solution $\bGa^{k}=(\Gamma^{k}_{i,j})_{i,j=1}^N$ of the operator $\mathcal{L}_{{\lambda},{\mu}}+\rho_ek^2$ can be expressed as (cf.\cite{LLZ}):
	\begin{equation}\label{eq:ef}
		\bGa^{k}=\bGa^{k}_s + \bGa^{k}_p,
	\end{equation}
	where 
	\[
	\bGa^{k}_p=-\frac{1}{\rho_e k^2} \nabla \nabla G^{{k}_p} \quad \mbox{and} \quad  \bGa^{k}_s =\frac{1}{\rho_e k^2}({k}_s^2 \mathbf{I} + \nabla \nabla) G^{{k}_s},
	\]
	with $\mathbf{I}$ denoting the $N\times N$ identity matrix and $G^k$ given in \eqref{eq:fu_he}.
	In the last equation ${k}_s$ and ${k}_p$ are defined by
	\begin{equation}\label{eq:kpks}
		k_p :=\frac{k\sqrt{\rho_e}}{\sqrt{ 2\mu+\lambda }}, \ \ k_s:=\frac{k\sqrt{\rho_e}}{\sqrt{ \mu}}. 
	\end{equation}
	When $k=0$, $\bGa^0(\bx)$ is defined as follows
	\begin{equation}\label{eq:f30}
		\bGa^0(\bx)=-\frac{\gamma_1}{4\pi} \frac{1}{|\bx|}\mathbf{I} -\frac{\gamma_2}{4\pi} \frac{\bx \bx^\top }{|\bx|^3},
	\end{equation}
	and
	\begin{equation}
		\gamma_1=\frac{1}{2}\left( \frac{1}{\mu} + \frac{1}{2\mu+\lambda} \right) \quad \mbox{and} \quad \gamma_2=\frac{1}{2}\left( \frac{1}{\mu} - \frac{1}{2\mu+\lambda}\right).
	\end{equation}
	Then the single layer potential operator associated with the fundamental solution $\bGa^{\omega}$ is defined by 
	\begin{equation}\label{eq:sinl}
		\mathbf{S}_{\partial \Omega}^{\omega}[\bvarphi](\bx)=\int_{\partial \Omega} \bGa^{\omega}(\bx-\by)\bvarphi(\by)ds(\by), \quad \bx\in\mathbb{R}^N,
	\end{equation}
	for $\bvarphi\in H^{-1/2}(\partial \Omega)^N$. On the boundary $\partial \Omega$, the conormal derivative of the single layer potential satisfies the following jump formula
	\begin{equation}\label{eq:jump}
		T_{\nu}  \mathbf{S}_{\partial \Omega}^{\omega}[\bvarphi] |_{\pm}(\bx)=\left( \pm\frac{1}{2}\mathbf{I} + \mathbf{K}_{\partial \Omega}^{\omega,*}  \right)[\bvarphi](\bx) \quad \bx\in\partial \Omega,
	\end{equation}
	where
	\begin{equation}\label{eq:lanp}
		\mathbf{K}_{\partial \Omega}^{\omega,*}[\bvarphi](\bx)= \int_{\partial \Omega} T_{\nu_{\bx}} \bGa^{\omega} (\bx-\by)\bvarphi(\by)ds(\by).
	\end{equation}
	The operator $\mathbf{K}_{\partial \Omega}^{\omega,*}$ is called the (N-P) operator associated with the Lam\'e system. 
	
	With the above preparations, we next derive the integral reformulation of the transmission eigenvalue problem \eqref{eq:syst3}. The main result is stated as follows. 
	
	\begin{thm}\label{thm:integral1}
		Consider the following spectral problem associated with a system of integral operator equations:
		\begin{equation}\label{eq:defA}
			\mathcal{A}(k,\delta) [\Phi](\bx)=0, \quad \bx\in\partial \Omega,
		\end{equation}
		where 
		\begin{equation}\label{eq:int1n1}
			\mathcal{A}(k,\delta):= \left(
			\begin{array}{cc}
				-\frac{I}{2} + K^{k,*}_{\partial \Omega} &  - k^2 \nu\cdot {\mathbf{S}}_{\partial \Omega}^{k\tau}\medskip \\
				\delta \tau^2\nu S^k_{\partial \Omega} & -\frac{\mathbf{I}}{2} + {\mathbf{K}}^{k\tau,*}_{\partial \Omega}\\
			\end{array}
			\right),
			\quad 
			\Phi:= \left(
			\begin{array}{c}
				\varphi_b \\
				\bvarphi_e \\
			\end{array}
			\right),
		\end{equation}
		with the surface densities $\varphi_b\in H^{-1/2}(\partial \Omega), \bvarphi_e\in H^{-1/2}(\partial \Omega)^N$. Then there exists a $k\in\mathbb{R}_+$ such that \eqref{eq:defA} possesses a nontrivial solution $\Phi\in H^{-1/2}(\partial\Omega)\times H^{-1/2}(\partial\Omega)^N$ if and only if $k$ is a transmission eigenvalue to the system \eqref{eq:syst3}.
		
	\end{thm}
	
	\begin{proof}
		We first prove the sufficiency part. If $k$ is a transmission eigenvalue to the system \eqref{eq:syst3},
		by the layer potential theory presented above, the solution to the system \eqref{eq:syst3} can be written as
		\begin{equation}\label{eq:sol}
			\left\{
			\begin{array}{ll}
				v=S_{\partial \Omega}^{k}[\varphi_b](\bx), & \bx\in \Omega,\medskip \\
				\mathbf{u}=\mathbf{S}_{\partial \Omega}^{k\tau}[\bvarphi_e](\bx), &  \bx\in \Omega,
			\end{array}
			\right.
		\end{equation}
		where density functions $\varphi_b\in H^{-1/2}(\partial \Omega), \bvarphi_e\in H^{-1/2}(\partial \Omega)^N$. Moreover, the density functions $\Phi=(\varphi_b, \varphi_e)$ solve the equation \eqref{eq:defA}.
		Indeed, the equation \eqref{eq:defA} follows from the third condition and the fourth condition in the system \eqref{eq:syst3}, the fact $k\neq 0$ and the jump formulas \eqref{eq:ju_he} as well as \eqref{eq:jump}. Thus there exists a $k\in\mathbb{R}_+$ such that \eqref{eq:defA} possesses a nontrivial solution, which proves the sufficient part. 
		
		Next, we show the necessity part. If there exists a $k\in\mathbb{R}_+$ such that \eqref{eq:defA} possesses a nontrivial solution $\Phi=(\varphi_b, \varphi_e)\in H^{-1/2}(\partial\Omega)\times H^{-1/2}(\partial\Omega)^N$, the functions $(\mathbf{u}, v)$ represented in \eqref{eq:sol} will solve the system \eqref{eq:syst3}. Indeed, the functions $(\mathbf{u}, v)$ represented in \eqref{eq:sol} satisfy the first and the second condition in \eqref{eq:syst3} by the potential theory. The third and the fourth equations follow from $\Phi$ solving the equation \eqref{eq:defA}. Thus all the equations in \eqref{eq:syst3} are fulfilled and this completes the proof.

	\end{proof}

	\subsection{Existence of AE transmission eigenvalues}
	
	In this subsection, {we show that there exist AE transmission eigenvalues to the system \eqref{eq:syst3}. }
	To that end, we firstly present some preliminaries for the subsequent analysis.  
	
	\begin{lem}\label{lem:ash}
		The operators $S_{\partial \Omega}^{k}: H^{-1/2}(\partial \Omega) \rightarrow H^{1/2}(\partial \Omega)$ and $K_{\partial \Omega}^{k,*}: H^{-1/2}(\partial \Omega)\rightarrow H^{-1/2}(\partial \Omega)$ defined in \eqref{eq:s_h} and \eqref{eq:ju_he} enjoy the following asymptotic expansion for $k\ll 1$ :
		\begin{equation}\label{eq:hsa}
			S_{\partial \Omega}^{k} = S_{\partial \Omega}^{0}  + k P , \qquad 
			K_{\partial \Omega}^{k,*}  =  K_{\partial \Omega}^{0,*}   + k^2 Q,
		\end{equation}
		where $S_{\partial \Omega}^{0}$ and $K_{\partial \Omega}^{0,*}$ are defined in \eqref{eq:s_h} and \eqref{eq:ju_he} with $k=0$, and the operators $P \in \mathcal{L}(H^{-1/2}(\partial \Omega), H^{1/2}(\partial \Omega))$, $Q\in\mathcal{L}(H^{-1/2}(\partial \Omega))$.
	\end{lem}
	\begin{proof} 
		For $k\ll 1$, the three dimensional fundamental solution given in \eqref{eq:fu_he} satisfies the following asymptotic expansion
		\begin{equation}\label{eq:asfh}
			\displaystyle{-\frac{e^{\mathrm{i} k|\bx|}}{4\pi |\bx|}}=\sum_{j=1}^{\infty} -\frac{\bsi}{4\pi} \frac{k^j(\bsi|\bx|)^{j-1}}{j!}.
		\end{equation}
		Thus, from the definition of the operator $S_{\partial \Omega}^{k}$ in \eqref{eq:s_h}, one has that 
		\[
		S_{\partial \Omega}^{k} = S_{\partial \Omega}^{0}  + k P,
		\]
		where the operator $P$ is given by
		\[
		P[\varphi](\bx) = -\frac{\bsi}{4\pi}\sum_{j=1}^{\infty} \int_{\partial D}\frac{( \bsi k |\bx-\by|)^{j-1}}{j!}\varphi(\by)ds(\by),
		\]
		and is a bounded operator from $H^{-1/2}(\partial \Omega)$ to $H^{1/2}(\partial \Omega)$ (cf. \cite{LLL}).
		
		For the N-P operator $K_{\partial \Omega}^{k,*}$, one has that  
		\[
		K_{\partial \Omega}^{k,*}  =  K_{\partial \Omega}^{0,*}   + k^2 Q,
		\]
		where the operator $Q$ is written by
		\[
		Q[\varphi](\bx) =\sum_{j=2}^{\infty} -\frac{(\bsi k)^j(j-1)}{4\pi j!}\int_{\partial D} |\bx-\by|^{j-3}(\bx-\by)\cdot\nu_\bx \varphi(\by) ds(\by),
		\]
		and is a bounded operator from $H^{-1/2}(\partial \Omega)$ to $H^{-1/2}(\partial \Omega)$ (cf. \cite{LLZ}). This completes the proof.
		
	\end{proof}
	
	\begin{rem}
		We only give the proof for the three dimensional case in Lemma \ref{lem:ash}. The two dimensional situation can be proved similarly. The same principle is also applied to Lemmas \ref{lem:asl} and \ref{lem:kerl}.
	\end{rem}

	\begin{lem}\label{lem:asl}
		For the operators $\mathbf{S}_{\partial \Omega}^{k}: H^{-1/2}(\partial \Omega)^N \rightarrow H^{1/2}(\partial \Omega)^N$ and $\mathbf{K}_{\partial \Omega}^{k,*}: H^{-1/2}(\partial \Omega)^N \rightarrow H^{-1/2}(\partial \Omega)^N$ defined in \eqref{eq:sinl} and \eqref{eq:lanp}, respectively, the following asymptotic expansions hold: 
		\begin{equation}\label{eq:hsa}
			\mathbf{S}_{\partial \Omega}^{k} = \mathbf{S}_{\partial \Omega}^{0} + k \mathbf{P} , \qquad 
			\mathbf{K}_{\partial \Omega}^{k,*}  = \mathbf{K}_{\partial \Omega}^{0,*} + k^2 \mathbf{Q},
		\end{equation}
		where $\mathbf{S}_{\partial \Omega}^{0}$ and $\mathbf{K}_{\partial \Omega}^{0,*}$ are defined in \eqref{eq:sinl} and \eqref{eq:lanp} with $k=0$, and the operator $\mathbf{P} \in \mathcal{L}\left(H^{-1/2}(\partial \Omega)^N, H^{1/2}(\partial \Omega)^N \right)$ and $\mathbf{Q}\in\mathcal{L}\left(H^{-1/2}(\partial \Omega)^N\right)$.
	\end{lem}
	
	\begin{proof}
		From the expression of the fundamental solution $\bGa^{\omega}$ in \eqref{eq:ef} and the asymptotic expansion in \eqref{eq:asfh}, one has that 
		\[
		\bGa^{k}(\bx)  = \sum_{j=0}^{\infty} k^j \bGa_j (\bx),
		\]
		where 
		\[
		\begin{split}
			\mathbf{\Gamma}_j(\bx) = & -\frac{1}{4\pi} \frac{\mathrm{i}^j}{(j+2)j!} \left( (j+1)\left( \frac{\rho_e}{\lambda + 2\mu} \right)^{\frac{j+2}{2}}  + \left( \frac{\rho_e}{ \mu} \right)^{\frac{j+2}{2}}  \right)  |\bx|^{j-1}\mathbf{I} \\
			& +\frac{1}{4\pi} \frac{\mathrm{i}^j(j-1)}{(j+2)j!} \left( \left( \frac{\rho_e}{\lambda + 2\mu} \right)^{\frac{j+2}{2}}  - \left( \frac{\rho_e}{ \mu} \right)^{\frac{j+2}{2}}  \right) |\bx|^{j-3}\bx\bx^T.
		\end{split}
		\]
		Thus the single layer potential operator $ \mathbf{S}_{\partial \Omega}^{k}$ can be written as 
		\[
		\mathbf{S}_{\partial \Omega}^{k} = \mathbf{S}_{\partial \Omega}^{0} + k \mathbf{P},
		\]
		where the operator $\mathbf{P}$ is given by 
		\[
		\mathbf{P}[\bvarphi](\bx)=\sum_{j=1}^{\infty}k^{j-1}\int_{\partial \Omega} \bGa_j(\bx-\by)\bvarphi(\by)ds(\by),
		\]
		and is bounded from $H^{-1/2}(\partial \Omega)^3$ to $H^{1/2}(\partial \Omega)^3$  (cf. \cite{DLL1}).
		
		For the N-P operator $ \mathbf{K}_{\partial \Omega}^{k,*}$, one has that  
		\[
		\mathbf{K}_{\partial \Omega}^{k,*}  = \mathbf{K}_{\partial \Omega}^{0,*} + k^2 \mathbf{Q}
		\]
		where the operator $\mathbf{Q}$ is written by
		\[
		\mathbf{Q}[\bvarphi](\bx) =\sum_{j=2}^{\infty} k^{j-2} \int_{\partial \Omega} T_{\nu_{\bx}} \bGa_j (\bx-\by)\bvarphi(\by)ds(\by),
		\]
		and is bounded operator from $H^{-1/2}(\partial \Omega)^3$ to $H^{-1/2}(\partial \Omega)^3$ (cf. \cite{DLL1}). This completes the proof.

	\end{proof}

	\begin{lem}\label{lem:kerh}
		The dimension of the kernel of the operator $-1/2 + K_{\partial \Omega}^{0,*}$ is $1$.
	\end{lem}
	\begin{proof}
		Assume that $\varphi$ is the kernel of the operator $-1/2 + K_{\partial \Omega}^{0,*}$. Thus the function $u(\bx)$ expressed by 
		\begin{equation}\label{eq:sv}
			v(\bx) = S_{\partial \Omega}^{0}[\varphi], \quad \bx\in\Omega,
		\end{equation}
		solves the following Neumann boundary value problem
		\begin{equation}\label{eq:nbp}
			\begin{cases}
				\Delta v(\mathbf{x})=0 &  \mbox{in}\quad \Omega ,\\
				\nabla v(\mathbf{x})\cdot \nu =0 & \mbox{on}\quad \partial \Omega.
			\end{cases}
		\end{equation}
		The dimension of the nontrivial solution to the system \eqref{eq:nbp} is $1$. Indeed, by the variational principle, the nontrivial solution to the equation \eqref{eq:nbp} is $v(\bx)=c$ with $c$ denoting an arbitrary constant. 
		
		Moreover, if $v(\bx)$ solves the system \eqref{eq:nbp}, the function $v(\bx)$ can be written in the form of \eqref{eq:sv}, which follows from that the operator $S_{\partial \Omega}^{0}$ is invertible from  $H^{-1/2}(\partial \Omega)$ to $H^{1/2}(\partial \Omega)$. The density function $\varphi$ will be the kernel of the operator $-1/2 + K_{\partial \Omega}^{0,*}$ due to $v(\bx)$ satisfying the second equation in \eqref{eq:nbp} and the jump formula \eqref{eq:ju_he}. The proof is complete.
	\end{proof}

	\begin{lem}\label{lem:kerl}
		The kernel of the operator $-1/2 +  \mathbf{K}_{\partial \Omega}^{0,*} $ is nontrivial. The dimension of the kernel is $3$ for $N=2$ and $6$ for $N=3$ (cf. \cite{bk:ab}).
	\end{lem}
	\begin{proof}
		Assume that $\bvarphi$ is the kernel of the operator $-1/2 +  \mathbf{K}_{\partial \Omega}^{0,*} $. Thus the function $\mathbf{u}(\bx)$ expressed by 
		\begin{equation}\label{eq:sve}
			\mathbf{u}(\bx) = \mathbf{S}_{\partial \Omega}^{0}[\bvarphi], \quad \bx\in\Omega,
		\end{equation}
		solves the following Neumann boundary value problem
		\begin{equation}\label{eq:nbpe}
			\begin{cases}
				\mathcal{L}_{\lambda, \mu} \mathbf{u}(\mathbf{x})=0 &  \mbox{in}\quad \Omega ,\\
				T_{\nu} \mathbf{u}(\mathbf{x}) =0 & \mbox{on}\quad \partial \Omega.
			\end{cases}
		\end{equation}
		The dimension of the nontrivial solution to the system \eqref{eq:nbpe} is $6$ for $N=3$ (cf. \cite{bk:ab}). 
		
		Moreover, if $\bu(\bx)$ solves the system \eqref{eq:nbpe}, the function $\bu(\bx)$ can be written in the form of \eqref{eq:sve}, which follows from that the operator $\mathbf{S}_{\partial \Omega}^{0}$ is invertible from  $H^{-1/2}(\partial \Omega)$ to $H^{1/2}(\partial \Omega)$. The density function $\bvarphi$ will be the kernel of the operator $-1/2 + \mathbf{K}_{\partial \Omega}^{0,*}$ due to $\bu(\bx)$ satisfying the second equation in \eqref{eq:nbpe} and the jump formula \eqref{eq:jump}. The proof is complete.
	\end{proof}

	For our use in what follows, we give the definition for the characteristic value \cite{AKL}. 
	\begin{defn}
		Denote by $B$ a Banach space. Let $\mathcal{U}_{z_0}$ be the set of all operator-valued functions with values in $\mathcal{L}(B)$ which are holomorphic in some neighborhood of $z_0$, except possibly at $z_0$. The point $z_0$ is called a {\it characteristic value} of $A(z)\in \mathcal{U}_{z_0}$ if there exists a vector-valued function $\varphi(z)$ with values in $B$ such that
		\begin{enumerate}
			\item $\varphi(z)$ is holomorphic at $z_0$ and $\varphi(z_0)\neq 0$,
			\item $A(z)\varphi(z)$ is holomorphic at $z_0$ and vanishes at this point.
		\end{enumerate}
	\end{defn}
	
	Based on the definition above and Theorem \ref{thm:integral1}, one can conclude the following proposition.
	\begin{prop}
		The existence of the transmission eigenvalues for the system \eqref{eq:syst3} is equivalent to the existence of a characteristic value $k$ to the operator $\mathcal{A}(k,\delta)$ defined in \eqref{eq:defA}. 
	\end{prop}
	
	In the next theorem, we show the existence of characteristic values to the operator $\mathcal{A}(k,\delta)$.
	
	\begin{thm}\label{thm:exei}
		Assume that the parameters introduced in Subsection \ref{sub:pa} satisfy 
		\begin{equation}\label{eq:asp}
			k=o(1),\ \ \delta=o(1),\ \ \tau=\mathcal{O}(1),\ \ \mu=\mathcal{O}(1),\ \ \lambda=\mathcal{O}(1).
		\end{equation}
		Then, for any small $\delta$, there exists a characteristic value $k^*(\delta)\ll 1$, which depends on $\delta$, to the operator-valued analytic function $\mathcal{A}(k,\delta)$; that is, there exists a nontrivial pair $(k^*, \Phi^*)$ such that $\mathcal{A}(k^*,\delta)[\Phi^*]=0$.
	\end{thm}
	\Blu{
	\begin{rem}
	We would like to point out that in \eqref{thm:exei}, for any given small $\delta$, we only prove the existence of at least one AE transmission eigenvalue to \eqref{eq:syst2}, though it is generically believed that there should be infinitely many eigenvalues. 
	\end{rem}
     }

	\begin{proof}[Proof of Theorem~\ref{thm:exei}]
		From the assumption in \eqref{eq:asp}, the operator $\mathcal{A}(k,\delta)$ has the following asymptotic expansion
		\[
		\mathcal{A}(k,\delta) = \mathcal{A}_0 + \mathcal{B},
		\]
		where 
		\[
		\mathcal{A}_{0} = \left(
		\begin{array}{cc}
			-\frac{I}{2} + K^{*}_{\partial \Omega} &  0  \medskip \\
			0 & -\frac{I}{2} + {\mathbf{K}}^{*}_{\partial \Omega}\\
		\end{array}
		\right),
		\]
		and $\mathcal{B} = \mathcal{O}(k^2 + \delta)$ following from Lemmas \ref{lem:ash} and \ref{lem:asl}. Moreover, the kernel of the operator $\mathcal{A}_0$ is nontrivial. Indeed, Lemmas \ref{lem:kerh} and \ref{lem:kerl} show that the dimension of the kernel for the operator $\mathcal{A}_0$ is  $3$ in two dimensions and $6$ in three dimensions. Thus $0$ is a characteristic value of the operator $\mathcal{A}(k, 0)$. Since the operator $\mathcal{A}(k,\delta)$ is a Fredholm operator, one can find a curve $\mathcal{C}$ in $\mathbb{C}$ that encloses the origin point $0$, such that the operator $\mathcal{A}(k,0)$ is invertible for $k\in\mathcal{C}$. By the Gohberg-Sigal theory \cite{AKL}, there exists a characteristic value $k^*(\delta)\ll 1$, which is located in the region enclosed by the curve $\mathcal{C}$, to the operator-valued analytic function $\mathcal{A}(k,\delta)$. This completes the proof. 
	\end{proof}
	
	\begin{rem}
		The assumption in \eqref{eq:asp} states that the size of the object $\Omega$ is smaller compared with the wavelengths of all the waves in the system: the acoustic wave and the compressional part as well as the shear part the elastic wave. In such a case, the air bubble $(\Omega; \kappa, \rho_e)$ is known as a nano-bubble, which is widely used in the composite elastic medium theory (cf. \cite{LS}). 
	\end{rem}
	
	 {Finally, we would like to mention that in the next section, we shall study the boundary-localization of the AE transmission eigenfunctions, where as a byproduct we show that within the radial geometry and no restriction on the medium parameters, there are infinitely many AE transmission eigenvalues.  }

	\section{Boundary localization of transmission eigenfunctions}\label{sect:3}

	In this section, we study the boundary-localization of the AE transmission eigenfunctions in 2D and 3D. Let us first focus on the radial case, i.e. we consider the acoustic-elastic transmission eigenvalue problem (\ref{eq:syst3}) with $ \Omega $ being a ball in $ \mathbb{R}^N $, $N=2,3$. Since $\Delta$ and $\mathcal L_{\lambda,\mu}$ are invariant under rigid motions and also noting \eqref{eq:para}, we can assume that $\Omega$ is the unit ball, namely $ \Omega:=\{\mathbf x\in\mathbb{R}^N; |\mathbf x|<1\} $. Define
	\begin{equation}\label{eq:Omega_tau}
		\Omega_{\epsilon}:=\{\mathbf x\in\mathbb{R}^N; |\mathbf x|<\epsilon\},\ \epsilon\in(0,1).
	\end{equation}
	\begin{defn}\label{def2.1}
		Consider a function $ \phi\in L^2(\Omega) $. It is said to be boundary-localized if there exists $ \epsilon\ll 1 $ such that 
		\begin{equation}\label{eq:surface1}
			\frac{\|\phi\|_{L^2(\Omega_{\epsilon})}}{\|\phi\|_{L^2(\Omega)}}\ll 1.
		\end{equation}
	\end{defn}
	\begin{rem}
		It is clear that if $ \phi\in L^2(\Omega) $ is boundary-localized, its $  L^2 $-energy concentrate on a small neighbourhood of $ \partial\Omega $, namely $ \Omega\backslash \Omega_{\epsilon} $. In what follows, we shall make the asymptotic expressions in Definition~\ref{def2.1} more rigorous. In fact, we shall construct a sequence of eigenfunctions $\{\phi_m\}_{m\in\mathbb{N}}$ such that for any given $\epsilon\in (0, 1)$, one has
		\begin{equation}\label{eq:surface1n1}
			\lim_{m\rightarrow\infty}\frac{\|\phi_m\|_{L^2(\Omega_{\epsilon})}}{\|\phi_m\|_{L^2(\Omega)}}=0.
		\end{equation}
		In such a case, we simply refer to $\{\phi_m\}$ as boundary-localized. 
	\end{rem}
	
	By noting \eqref{eq:para}, it is remarked that the p-wavenumber and s--wavenumber associated with the Lam\'e operator in \eqref{eq:syst3} satisfy that
	\begin{equation}\begin{split}\label{eq:k_p}
			k_p&
			=\frac{k\tau}{\sqrt{\lambda+2\mu}}=k\tau,\quad k_s
			=\frac{k\tau}{\sqrt{\mu}},
	\end{split}\end{equation}
	where $k, \lambda, \mu$ and $\tau$ are the parameters in \eqref{eq:syst3}. \Blu{ It is remarked that in \eqref{eq:k_p} we have used the relation $\sqrt{\lambda + 2\mu}=1$ that follows from \eqref{eq:para}.}
	\Blu{Moreover, it is noted that the wave speed in elastic media is larger than that in fluid media, which means that $\tau<1$ from \eqref{eq:node}. Thus in what follows, we shall confine our study on the case $\tau\in (0, 1)$ for physical relevance.}
		
	In the sequel, we let $ m\in\mathbb{N} $ be a positive integer, $ J_m(|\mathbf x|) $ be the first kind Bessel function of order $m$, and $J'_m(|\mathbf x|)  $ be the derivative of $ J_m(|\mathbf x|) $. Furthermore, we let $ j_{m,s} $ denote the $s$-th positive zero of $ J_m(|\mathbf x|) $, in the mean while, $ j'_{m,s} $ denote the $s$-th positive zero of $ J'_m(|\mathbf x|) $. From \cite{Abr}, one has 
	\begin{equation}
		\begin{split}
			&m\le j'_{m,1}<j_{m,1}<j'_{m,2}<j_{m,2}<j'_{m,3}<j_{m,3}<\ldots,\label{eq:jm root} \ J_m(|x|)=\frac{(|x|/2)^m}{\Gamma(m+1)}\prod_{s=1}^{\infty}\left(1-\frac{|x|^2}{j^2_{m,s}}\right). 
		\end{split}
	\end{equation}
	In order to prove the result of boundary-localization, we assume that the order $ m $ of the Bessel function $J_m(x)  $ is sufficiently large.
	
	\subsection{Two-dimensional results}
	In this subsection, we mainly prove the results of boundary-localization in $\mathbb R^2$. Let $ \mathbf{x}:=(x_1,x_2)=(r\cos\theta, r\sin\theta)\in\mathbb{R}^2 $ denote the polar coordinate. Using Fourier expansion, the acoustic-elastic transmission eigenfunctions to \eqref{eq:syst3} associated with the transmission eigenvalues $k\in \mathbb R_+$ have series expansions as follow (cf.\cite{SP,CDLZ,ColtonKress,DLW,DLW22}):
	\begin{equation}\begin{split}\label{eq:uv}
			\mathbf{u}=&\sum_{m=0}^{+\infty}\Biggl\{\alpha_m\left( k_pJ'_m(k_p|\mathbf{x}|)e^{\mathrm{i}m\theta}\cdot\hat{r}+\frac{\mathrm{i}m}{|\mathbf{x}|}J_m(k_p|\mathbf{x}|)e^{\mathrm{i}m\theta}\cdot\hat{\theta}\right)\\
			&+\gamma_m\left(\frac{\mathrm{i}m}{r}J_m(k_s|\mathbf{x}|)e^{\mathrm{i}m\theta}\cdot\hat{r}-k_sJ'_m(k_s|\mathbf{x}|)e^{\mathrm{i}m\theta}\cdot\hat{\theta}\right)\Biggr\},\\
			v=&{\sum_{m=0}^{+\infty}\beta_me^{\mathrm{i}m\theta}   J_m(k|\mathbf{x}|),}
		\end{split}
	\end{equation}
	where $\alpha_m$, $\beta_m$ and $\gamma_m$are complex constants,  $ k_p $ is given in (\ref{eq:k_p}), and 
	\begin{equation}\label{hat}
		\hat{r}:=\left(\begin{array}{c}
			\cos\theta \\
			\sin\theta
		\end{array}
		\right)\ \quad \mbox{and} \quad   \hat{\theta}:=\left(\begin{array}{c}
			-\sin\theta \\
			\cos\theta
		\end{array}
		\right).
	\end{equation}
				
			
	Next, we show that there exists a sequence of discrete  transmission eigenvalues of \eqref{eq:syst3}. Moreover,  we prove that infinity is the  only accumulation point of this sequence. 
	
	\begin{lem}\label{lem2.1}
		Consider \eqref{eq:syst3}, where $ \Omega\in\mathbb{R}^2 $ is  the unit  disk and the parameter $\tau$ is chosen such that { $\tau<j_{m,1}/ j_{m,2}<1$}.  Let $ \{k_{m,\ell}~|~\ell=1,2,\ldots\}$ be the subset of AE transmission eigenvalues $k$  of (\ref{eq:syst3}), where $ m $ is the order of the Bessel functions and $\ell $ denotes the $\ell$-th eigenvalue for a fixed $m$. 
		Then there exists a subsequence of $ \{k_{m,\ell}\} $, denoted by $ \{k_{m,s(m)}\} $, such that for $ m $
		sufficiently  large, it holds that 
		\begin{equation}\label{kms}
			k_{m,s}\in \left(j_{m,1},\ j_{m,2}\right),
		\end{equation}	
		 Furthermore, it yields that
		\begin{equation}\label{con1}
			\frac{	k_{m,s}}{m}\to 1,\quad \mbox{as}\quad m\to+\infty.
		\end{equation}
		and 
		\begin{equation}\label{con2}
			k_{m,s}=m\left(1+C_0m^{-\frac{2}{3}}+o(m^{-\frac{2}{3}})\right),
		\end{equation}
		where $C_0$ is a positive constant not depending on $m$. 
	\end{lem}
	\begin{proof}
		Let $k$ be an AE transmission eigenvalue of \eqref{eq:syst3}. 
		Recall that $k_p$ is defined in \eqref{eq:k_p}.  For a fixed $ m\in\mathbb{N}_+ $,  using Fourier expansion, it is not difficult to see that \begin{equation}\label{eq:u,v}
			\begin{split}
				\mathbf{u}_m(\mathbf{x})&=\alpha_m k_pJ'_m(k_p|\mathbf{x}|)e^{\mathrm{i}m\theta}\cdot\hat{r}+\alpha_m\frac{\mathrm{i}m}{|\mathbf{x}|}J_m(k_p|\mathbf{x}|)e^{\mathrm{i}m\theta}\cdot\hat{\theta}, \ 
				v_m(\mathbf{x})=\beta_mJ_m(k|\mathbf{x}|)e^{\mathrm{i}m\theta},
			\end{split}
		\end{equation}
		where $ \alpha_m $ and $ \beta_m $ are nonzero constants, 
		is a pair solution to  $ \mathcal{L}_{\lambda, \mu}\bmf {u}(\mathbf{x}) + k^2\tau^2\bmf {u}(\mathbf{x})= \mathbf  0$ and $ \Delta v(\mathbf{x})+k^2v(\mathbf{x})=0$  in $ \Omega $.  Hence, by the first transmission condition of (\ref{eq:syst3}) on $ \partial\Omega $, one can arrive at that
		\begin{equation}\notag 
			\bmf {u}_m(\mathbf{x})\cdot \nu-\frac{1}{k^2}\nabla v_m(\mathbf{x})\cdot \nu =\bmf {u}_m\cdot\hat{r}-\frac{1}{k^2}\frac{\partial v_m}{\partial r} = \alpha_m k_p J'_m(k_pr)e^{\mathrm{i}m\theta}-\frac{1}{k}\beta_m J'_m(kr)e^{\mathrm{i}m\theta}=0.
		\end{equation}
		It is ready to know that
		\begin{equation}\label{betam}
			\beta_m=\alpha_m k_p k\frac{J'_m(k_p)}{J'_m(k)}.
		\end{equation} 	it is obvious that $\alpha_m$ and  $\beta_m$ are not zero by contradiction.
		By direct calculation, we can obtain that 
		\begin{equation}\label{grad u}
			\nabla\cdot\bmf{u}_m
			=
			\alpha_mk_p^2J''_m(k_pr)e^{\mathrm{i}m\theta}+\alpha_m\frac{k_p}{r}J'_m(k_pr)e^{\mathrm{i}m\theta}-\alpha_m\frac{m^2}{r^2}J_m(k_pr)e^{\mathrm{i}m\theta},
		\end{equation}
		and 
		\begin{equation}
			\begin{split}\label{partial1}
				\frac{\partial u_{m,1}}{\partial x_2}=&\alpha_mk_p^2J''_m(k_pr)e^{\mathrm{i}m\theta}\cos\theta\sin\theta+\alpha_m\frac{\bsi m}{r^2}J_m(k_pr)e^{\mathrm{i}m\theta}\sin^2\theta-\alpha_m\frac{\bsi m}{r}k_p J'_m(k_pr) \\
				&\times e^{\mathrm{i}m\theta} \sin^2\theta +\alpha_m\frac{\bsi m}{r}k_p J'_m(k_pr) e^{\mathrm{i}m\theta}\cos^2\theta-\alpha_m\frac{k_p}{r} J'_m(k_pr) e^{\mathrm{i}m\theta}\sin\theta\cos\theta\\
				&+\alpha_m\frac{m^2}{r^2} J_m(k_pr) e^{\mathrm{i}m\theta}\sin\theta\cos\theta-\alpha_m\frac{\bsi m}{r^2}J_m(k_pr)e^{\mathrm{i}m\theta}\cos^2\theta=\frac{\partial u_{m,2}}{\partial x_1}.
			\end{split}
		\end{equation}
		By using the second transmission condition of (\ref{eq:syst3}) on $ \partial\Omega $, one has
		\begin{equation}\label{bound2}
			\begin{split}
				&\left(\lambda(\nabla\cdot\bmf{u}_m)I_2+2\mu)	\nabla^s\bmf{u}_m+\delta\tau^2 v_m\right)\cdot\hat{r}\\
				=&\lambda\left(\nabla\cdot\bmf{u}_m\right)\left(\begin{array}{c}
					\cos\theta \\
					\sin\theta
				\end{array}
				\right)+2\mu\left(\begin{array}{c}
					\frac{\partial u_{m,1} }{\partial x_1}\cos\theta+\frac{\partial u_{m,1} }{\partial x_2}\sin\theta \\
					\frac{\partial u_{m,2} }{\partial x_1}\cos\theta+\frac{\partial u_{m,2}}{\partial x_2}\sin\theta
				\end{array}
				\right)+\delta\tau^2 v_m\left(\begin{array}{c}
					\cos\theta \\
					\sin\theta
				\end{array}
				\right)=0. 
			\end{split}
		\end{equation}
		{	Comparing the coefficients of the functions $\sin\theta$ and $\cos\theta$ in \eqref{bound2} and combining equations \eqref{eq:u,v} to \eqref{partial1},
			there hold
			\begin{equation}\label{rela1}
				J_m(k_p)=k_pJ'_m(k_p),
			\end{equation}
			and
			\begin{equation}\label{rela1.5}
				-k_p^2J_m(k_p)J'_m(k)-2\mu k_pJ'_m(k_p)J'_m(k)+2\mu m^2J_m(k_p)J'_m(k)+\delta\tau^2kk_pJ'_m(k_p)J_m(k)=0.
			\end{equation} 
			Now, we want to find a $k$ that satisfies both (\ref{rela1}) and (\ref{rela1.5}). Instead of directly solving the two equations, we first substitute (\ref{rela1}) into (\ref{rela1.5}), one has 
			\begin{equation}\label{rela2}
				\left[-k_p^2J'_m(k)+2\mu m^2 J'_m(k)-2\mu J'_m(k)+\delta\tau^2 k J_m(k)\right] J_m(k_p)=0.
			\end{equation}
			Apparently, if $k$ solves the equations \eqref{rela1} and \eqref{rela1.5}, the equation \eqref{rela2} holds. On the other hand, if $k$ is the root of the equations \eqref{rela1} and \eqref{rela2}, $k$  solves the equation \eqref{rela1.5}. Thus, in the rest of the proof, we shall confine ourselves on finding the root $k$ of the equation \eqref{rela2}. Then we will show that the root $k$ also solves \eqref{rela1} in the asymptotic sense and please refer to Remark \ref{rem:as18}.
		}
		
		By utilizing $ J'_m(|\mathbf{x}|)=J_{m-1}(|\mathbf{x}|)-\frac{m}{|\mathbf{x}|}J_{m}(|\mathbf{x}|) $ and ($ \ref{eq:k_p} $),  the root $ k $ of \eqref{rela2} coincides with that of the following function:
		\begin{equation}\begin{split}
				\label{rela3}
				f(k):=&\left[-k^2\tau^2+2\mu(m^2-1)\right]J_{m-1}(k)J_{m}(\tau k)\\
				&+\Big[k\tau^2m-\frac{2\mu(m^2-1)m}{k}+\delta\tau^2k\Big]J_{m}(k)J_{m}(\tau k). 
		\end{split}\end{equation}
		
		

		Next, we compute the following quantity:
		\begin{equation}\begin{split}\label{f}
				f_m(j_{m,1})f_m(j_{m,2})=&\left[-j_{m,1}^2\tau^2+2\mu(m^2-1)\right]\left[-j_{m,2}^2\tau^2+2\mu(m^2-1)\right]\\
				&\times J_{m-1}(j_{m,1})J_{m-1}(j_{m,2})J_{m}(\tau j_{m,1})J_{m}(\tau j_{m,2})\\
			\end{split}
		\end{equation}

		Suppose that $a_s$ is the $s$th negative zero of the Airy function. From   \cite{FWJO10},  we have   $a_s=-[3\pi/8(4s-1)]^{2/3}(1+\Upsilon_s ) $, where $\Upsilon_s\in [0,\,  0.13(3\pi/8(4s-1.051) )^{-1}] $. The positive root of the Bessel function can have the following sharp lower and upper bound. Indeed, from \cite{quwong} one knows that
		\begin{equation}\label{eq:jms up}
			m\left(1- \frac{ a_s }{(2m^2)^{1/3}}  \right)	<j_{m,s} <m\left(1-\frac{ a_s }{(2m^2)^{1/3}} +\frac{3a_s^2}{20}\left(\frac{2}{m^4} \right)^{1/3} \right),
		\end{equation}
		where $j_{m,s}$ is the $s$th positive root of $J_m(|\mathbf x|)$. Hence,  when $ m $ is sufficiently large, there hold that
		\begin{equation}\begin{split}\label{j_m}
				j_{m,1}&=m\Big(1+C_1m^{-2/3}+o(m^{-2/3})\Big),\\
				j_{m,2}&=m\Big(1+C_2m^{-2/3}+o(m^{-2/3})\Big),
			\end{split}
		\end{equation} where $ C_i,i=1,2 $ is a positive constant not depending on $m$.  By virtue of \eqref{j_m}, when $m$ is sufficient large, we have 
		\begin{equation}\begin{split}\label{eq:1}
				&\left[-j_{m,1}^2\tau^2+2\mu(m^2-1)\right]\left[-j_{m,2}^2\tau^2+2\mu(m^2-1)\right]\\
				&=m^4\,\Pi_{i=1}^2\left[\tau^2-2\mu + \tau ^2 C_i  m^{-2/3 }+o(m^{-2/3}) +2\mu m^{-2} \right] >0
		\end{split}\end{equation}
		for any positive constants $\tau$ and $\mu$. 
		{
			Moreover, from the choice of $\tau j_{m,2}<j_{m,1}$, the following inequality holds
			\begin{equation}
				J_{m}(\tau j_{m,1})J_{m}(\tau j_{m,2})>0.
			\end{equation}
			Since the positive zeros of $J_{m-1} $ are interlaced with those of $J_m$, we have 
			\[
			J_{m-1}(j_{m,1})J_{m-1}(j_{m,2})<0.
			\]
			By using the above facts, we derive that 
			\begin{equation}\begin{split}\label{f}
					f_m(j_{m,1})f_m(j_{m,2})<0,
				\end{split}
			\end{equation} which implies that $k_{m,s}\in \left(j_{m,1},j_{m,2}\right) $.
			Finally,  using (\ref{j_m}) one has (\ref{con1}) and (\ref{con2}).	}
		The proof is complete.
	\end{proof}
	
	{\begin{rem}\label{rem:as18}
			Here we remark that if $k_{m,s}$ solves \eqref{rela2}, then $k_{m,s}$  solves the equation \eqref{rela1} in the asymptotic sense. 
			Indeed, if $k_{m,s}$ is a root of (\ref{rela2}), for sufficiently large $m$ and $\tau<j_{m,1}/j_{m,2}<1$, using (\ref{kms}),(\ref{con1}) and the following asymptotic expansion (cf.\cite{Korenev})
			\begin{equation}\label{asy1}
				J_m(k_{m,s}r)=\frac{z^me^{m\sqrt{1-z^2}}}{(2\pi m)^{1/2}(1-z^2)^{1/4}(1+\sqrt{1-z^2})^m}\left(1+o(1)\right),\  z=\frac{k_{m,s}r}{m},\ 0<r<1,
			\end{equation}
			we derive that
			\begin{equation}
				J_m(k_p)-k_pJ'_m(k_p)=Cm^\frac{1}{2}\left(\frac{\tau e^{\sqrt{1-\tau^2}}}{1+\sqrt{1-\tau^2}}\right)^{m-1}\rightarrow 0 \quad \mbox{as} \quad m\rightarrow \infty,
			\end{equation} 
			where C is a constant. In the derivation of the last inequality, we have used the fact $0<\frac{\tau e^{\sqrt{1-\tau^2}}}{1+\sqrt{1-\tau^2}}<1$ due to $\tau<1$ . Thus, $k_{m,s}$ solves (\ref{rela1}) in the asymptotic sense. 
		\end{rem}
	}
	
	Next, we shall prove that the AE transmission eigenfunction $v_m$ associated with transmission eigenvalue $k_{m,s}$ given by Lemma \ref{lem2.1} is boundary-localized. 
	
	\begin{thm}\label{th2.4}
		Consider the AE transmission eigenvalue problem (\ref{eq:syst3}). Let  $ \Omega $ be the unit disk in $ \mathbb{R}^2 $ and $ \Omega_{\epsilon} $ be given in (\ref{eq:Omega_tau}). For any fixed $ \epsilon\in(0,1) $, 
		there exist transmission eigenfunctions $ v_m $, $m\in\mathbb{N}$, associated with the eigenvalues $ k_{m,s} $ of (\ref{eq:syst3}) described in Lemma \ref{lem2.1} such that 
		\begin{equation}\label{eq:th:sur3}
			\lim_{ m \to \infty}\frac{\|v_m\|_{L^2(\Omega_{\epsilon})}}{\|v_m\|_{L^2(\Omega)}}=0.
		\end{equation}
	\end{thm}
	\begin{proof}
		Let the acoustic-elastic transmission eigenfunction $v_m$ associated with the transmission eigenvalue $k_{m,s}$ be given in the form (\ref{eq:u,v}), where $k_{m,s}$ fulfils \eqref{kms}. Hence by direct calculations, it yields that
		\begin{equation}\label{v_m1}
			\begin{split}
				\|v_m\|^2_{L^2(\Omega_{\epsilon})}&=|\beta_m|^2\int_{\Omega_{\epsilon}}|J_m(k_{m,s}|\mathbf{x}|)|^2 d\mathbf{x}=2\pi|\beta_m|^2\int_{0}^{\epsilon}rJ^2_m(k_{m,s}r) dr,\\
				\|v_m\|^2_{L^2(\Omega)}&=|\beta_m|^2\int_{\Omega}|J_m(k_{m,s}|\mathbf{x}|)|^2 d\mathbf{x}=2\pi|\beta_m|^2\int_{0}^{1}rJ^2_m(k_{m,s}r) dr.
			\end{split}
		\end{equation}

		Recall that $j'_{m,1}$ is the first zero of $J'_m(|\mathbf x|)$. According to \eqref{eq:jm root} and \eqref{eq:jms up}, it readily yields that
		\begin{equation}\label{eq:jpm1}
			j'_{m,1}=m(1+O(m^{-2/3})) \quad \mbox{when} \quad  m\rightarrow \infty. 
		\end{equation}
		Due to \eqref{con2} and \eqref{eq:jpm1}, for sufficiently large $m$, one can claim that
		$$
		k_{m,s} r<j'_{m,1},\quad r\in (0,1). 
		$$
		Denote $H_m(t)=J_m (k_{m,s}r)$, where $r\in (0,1)$ and $t=k_{m,s}r\in (0,j'_{m,1})$. By virtue of \eqref{eq:jm root}, $J_m(t)\geq 0$ for $t\in [0, j'_{m,1}]$ and  $j'_{m,1}$ is the first maximum point  of  $J_m (t)$. Hence $H_m(t)$ is monotonically increasing in $(0,1)$, which implies that  $ J'_m(k_{m,s}r) >0 $ for $ 0<r<1 $.

		
		Hence, using (\ref{asy1}) and (\ref{v_m1}), we can obtain the 	asymptotic upper bound for the integral as follows
		\begin{equation}
			\begin{split}\label{vmeps1}
				& \int_{0}^{\epsilon}rJ^2_m(k_{m,s}r) dr \leqslant \epsilon^2J^2_m(k_{m,s}\epsilon)\\
				&=\frac{\epsilon^2}{2\pi m\left(1-(\frac{k_{m,s}\epsilon}{m})^2\right)^{1/2}}\left(\frac{k_{m,s}\epsilon}{m} \frac{e^{\sqrt{1-(\frac{k_{m,s}\epsilon}{m})^2}}}{1+\sqrt{1-(\frac{k_{m,s}\epsilon}{m})^2}} \right)^{2m} \left(1+o(1)\right)^2,
			\end{split}
		\end{equation}
		when $m$ is sufficient large. Furthermore,  in view of \eqref{eq:jpm1}, we denote
		\begin{equation}\label{r1}
			\Xi =\frac{j'_{m,1}}{k_{m,s}}.
		\end{equation} 
		One  has $ \epsilon<\Xi<1 $ when  $ m $ is sufficiently large. Hence, it yields that $ m<j'_{m,1}=k_{m,s}r_1 $. { Using   \eqref{con2}, \eqref{eq:jpm1} and the following asymptotic formula: 
			\begin{equation}\label{J_m}
				J_m(|\mathbf{x}|)=\sqrt{\frac{2}{\pi\sqrt{|\mathbf{x}|^2-m^2}}}\cos\left(\sqrt{|\mathbf{x}|^2-m^2}-\frac{m\pi}{2}+m\arcsin(\frac{m}{|\mathbf{x}|})-\frac{\pi}{4}\right)\left(1+o(1)\right),
			\end{equation} for sufficient large $m$, we can deduce that
			\begin{equation}
				\begin{split}\label{vm1}
					\int_{0}^{1}rJ^2_m(k_{m,s}r) dr
					\geqslant& \int_{\Xi }^{1}rJ^2_m(k_{m,s}r) dr=\frac{2(1+o(1))^2}{\pi}\int_{\Xi}^{1}\frac{r\cos^2 x_m}{\left( k^2_{m,s}r^2-m^2\right)^{1/2}}dr\\
					\sim&\frac{1}{\pi}\int_{\Xi}^{1}\frac{r}{\left( k^2_{m,s}r^2-m^2\right)^{1/2}} dr\\
					\sim&\frac{1}{\pi k^2_{m,s}}\Big(\sqrt{ k_{m,s}^2-m^2}-\sqrt{j_{m,1}'^2-m^2}\Big).
			\end{split}\end{equation}
			where $x_m=\sqrt{{(\tau k_{m,s}r)}^2-m^2}-\frac{m\pi}{2}+m\arcsin(\frac{m}{\tau k_{m,s}r})-\frac{\pi}{4} $.

			Combining (\ref{vmeps1}) and  (\ref{vm1}), there holds
			\begin{equation}\begin{split}\label{compare1}
					\frac{\|v_m\|^2_{L^2(\Omega_{\epsilon})}}{\|v_m\|^2_{L^2(\Omega)}}\leqslant&Cm^{\frac{1}{3}}\left(1-\epsilon^2\right)^{-\frac{1}{2}}\left(\frac{\epsilon e^{\sqrt{1-\epsilon^2}}}{1+\sqrt{1-\epsilon^2}} \right)^{2m} 
			\end{split}\end{equation} 
			where $C$ is a constant. When $m$ is sufficiently large, let $ m\to\infty $ one has (\ref{eq:th:sur3}). }
		
		The proof is complete.	
	\end{proof}

	\begin{thm}\label{th2.5}
		Consider (\ref{eq:syst3}), where $ \Omega $ is the unit disk in $ \mathbb{R}^2 $ and $ \Omega_{\epsilon} $ is given in (\ref{eq:Omega_tau}).  If the physical parameter $\tau$ in  \eqref{eq:syst3} fulfilling $ \tau\in(0, 1) $, then for any $ \epsilon\in(0,1) $, there exists eigenfunctions $ \mathbf{u}_m $, $m\in\mathbb{N}$, associated with transmission  eigenvalues $ k_{m,s} $ described in Lemma \ref{lem2.1} such that 
		\begin{equation}\label{eq:th:sur4}
			\lim_{ m \to \infty}\frac{\|\mathbf{u}_m\|_{L^2(\Omega_{\epsilon})}}{\|\mathbf{u}_m\|_{L^2(\Omega)}}=0.
		\end{equation}
	\end{thm}
	
	\begin{proof} Suppose that the AE transmission eigenfunction ${\mathbf u}_m$ associated with the transmission eigenvalue $k_{m,s}$ is given in the form (\ref{eq:u,v}), where $k_{m,s}$ satisfies  \eqref{kms}.
		By direct calculations and \eqref{eq:k_p}, it yields that
		\begin{subequations}
			\begin{align}
				\|\mathbf{u}_m\|_{L^2(\Omega)} \label{um1}
				&=2\pi|\alpha_m|^2\int_{0}^{1}\left( k_p^2rJ_m'^{2}(k_pr)+\frac{m^2}{r}J_m^2(k_pr)\right)dr\\
				&=2\pi|\alpha_m|^2\int_{0}^{1}\left( \tau^2 k^2_{m,s}rJ_m'^{2}(\tau k_{m,s}r)+\frac{m^2}{r}J_m^2(\tau k_{m,s}r)\right) dr, \notag  \\
				\label{um2}
				\|\mathbf{u}_m\|_{L^2(\Omega_\epsilon)}&=2\pi|\alpha_m|^2\int_{0}^{\epsilon} \left(\tau^2 k^2_{m,s}rJ_m'^{2}(\tau k_{m,s}r)+\frac{m^2}{r}J_m^2(\tau k_{m,s}r) \right) dr.
			\end{align}
		\end{subequations} 
		From \cite{FWJO10}, one has
		\begin{subequations}
			\begin{align}
				\label{J'm}
				|J'_m(mx)|&\leqslant\frac{\left(1+x^2\right)^{1/4}}{x\left(2\pi m\right)^{1/2}}\frac{x^me^{m\sqrt{1-x^2}}}{\left(1+\sqrt{1-x^2}\right)^m},\quad \mbox{for}\quad  m>0\quad  \mbox{and}\quad 0<x\leqslant 1,\\
				|J_m(mx)|&\leqslant\frac{x^me^{m\sqrt{1-x^2}}}{\left(1+\sqrt{1-x^2}\right)^m},\quad \mbox{for}\quad m\geqslant0\quad \mbox{and} \quad 0<x\leqslant1.	\label{Jm(mx)}	
			\end{align}
		\end{subequations}
		Under the assumption $\tau\leqslant 1$, by virtue  of \eqref{con2}, if $m$ is sufficient large, it yields that
		$$
		\frac{\tau k_{m,s} r }{m} <1, \quad \forall r\in [0,\epsilon]. 
		$$
		Therefore, due to (\ref{J'm}) and (\ref{Jm(mx)}), after tedious calculations  we can derive that
		\begin{align}\label{int1}
			&\int_{0}^{\epsilon}\left( \tau^2 k^2_{m,s}rJ_m'^{2}(\tau k_{m,s}r)+\frac{m^2}{r}J_m^2(\tau k_{m,s}r)\right) dr\\
			\leqslant&\int_{0}^{\epsilon}\left(\frac{\sqrt{1+(\frac{\tau k_{m,s}r}{m})^{2}}}{2\pi}+m\right) \left(\frac{\tau k_{m,s}e^{\sqrt{1-(\frac{\tau k_{m,s}r}{m})^2}}}{1+\sqrt{1-(\frac{\tau k_{m,s}r}{m})^2}}\right)\left(\frac{\tau k_{m,s}r}{m}\frac{e^{\sqrt{1-(\frac{\tau k_{m,s}r}{m})^2}}}{1+\sqrt{1-(\frac{\tau k_{m,s}r}{m})^2}}\right)^{2m-1}dr\notag \\
			=&\epsilon\cdot\left(\frac{\sqrt{1+(\frac{\tau k_{m,s}\eta_1}{m})^{2}}}{2\pi}+m\right) \left(\frac{\tau k_{m,s}e^{\sqrt{1-(\frac{\tau k_{m,s}\eta_1}{m})^2}}}{1+\sqrt{1-(\frac{\tau k_{m,s}\eta_1}{m})^2}}\right)\left(\frac{\tau k_{m,s}\eta_1}{m}\frac{e^{\sqrt{1-(\frac{\tau k_{m,s}\eta_1}{m})^2}}}{1+\sqrt{1-(\frac{\tau k_{m,s}\eta_1}{m})^2}}\right)^{2m-1} \notag
		\end{align}
		where $ \eta_1\in[0,\epsilon] $.
		
		 When  $\tau < 1$, recall that $ \Xi  $ is given in (\ref{r1}), using  \eqref{asy1} we can obtain the following inequality 
		\begin{equation}\begin{split}\label{int2}
				&\int_{0}^{1}\left( \tau^2 k^2_{m,s}rJ_m'^{2}(\tau k_{m,s}r)+\frac{m^2}{r}J_m^2(\tau k_{m,s}r)\right) dr\\
				\geqslant
				&\int_{\Xi }^{1} \frac{m^2}{r}\frac{1}{2\pi m \sqrt{1-(\frac{\tau k_{m,s}r}{m})^2}}\left(\frac{\tau k_{m,s}r}{m}\frac{e^{\sqrt{1-(\frac{\tau k_{m,s}r}{m})^2}}}{1+\sqrt{1-(\frac{\tau k_{m,s}r}{m})^2}}\right)^{2m}dr\cdot(1+o(1))^2\\
				\geqslant&\int_{\Xi }^{1}\frac{1}{2\pi  \sqrt{1-(\frac{\tau k_{m,s}r}{m})^2}}\left(\frac{\tau k_{m,s}e^{\sqrt{1-(\frac{\tau k_{m,s}r}{m})^2}}}{1+\sqrt{1-(\frac{\tau k_{m,s}r}{m})^2}}\right)\left(\frac{\tau k_{m,s}r}{m}\frac{e^{\sqrt{1-(\frac{\tau k_{m,s}r}{m})^2}}}{1+\sqrt{1-(\frac{\tau k_{m,s}r}{m})^2}}\right)^{2m-1}dr\\
				=&\frac{1-\Xi }{2\pi  \sqrt{1-(\frac{\tau k_{m,s}\eta_2}{m})^2}}\left(\frac{\tau k_{m,s}e^{\sqrt{1-(\frac{\tau k_{m,s}\eta_2}{m})^2}}}{1+\sqrt{1-(\frac{\tau k_{m,s}\eta_2}{m})^2}}\right)\left(\frac{\tau k_{m,s}\eta_2}{m}\frac{e^{\sqrt{1-(\frac{\tau k_{m,s}\eta_2}{m})^2}}}{1+\sqrt{1-(\frac{\tau k_{m,s}\eta_2}{m})^2}}\right)^{2m-1},
			\end{split}
		\end{equation} { where $ \eta_2\in[r_1,1] $ and $ \eta_2\to 1 $ as $m\to\infty  $.
			
			Now, we introduce the auxiliary function 
			\begin{align}\label{eq:phi def}
				\phi(x)=\frac{xe^{\sqrt{1-x^2}}}{1+\sqrt{1-x^2}}, 
			\end{align}
			which is monotonically increasing for $ x\in (0,1) $. Noting that $ \lim_{ m \to \infty}\left(\frac{\tau k_{m,s}\eta_2}{m}-\frac{\tau k_{m,s}\eta_1}{m}\right)=\tau (1-\eta_1)>0 $. Therefore there exists $ \delta(\eta_1)\in (0,1) $ such that
			\begin{equation}\label{phi1}
				0<\frac{\phi\left(\frac{\tau  k_{m,s}\eta_1}{m}\right)}{\phi\left( \frac{\tau k_{m,s}\eta_2}{m}\right)}<1-\delta(\eta_1)
			\end{equation}
			for sufficiently large $m$.   Combining  (\ref{um1}), (\ref{um2}), (\ref{int1}) with (\ref{int2}), there holds
			\begin{equation}\label{compare2}
				\frac{\|\mathbf{u}_m\|_{L^2(\Omega_\epsilon)}}{\|\mathbf{u}_m\|_{L^2(\Omega)}}\leqslant C\epsilon m^{\frac{5}{3}}(1-\delta(\eta_1))^{2m-1},
		\end{equation} }
		where $ C $ is a constant.
		Finally, letting  $ m\to\infty $ one has (\ref{eq:th:sur4}).
		\medskip

		The proof is complete.
	\end{proof}

	\subsection{Three-dimensional results}
	In this subsection, we mainly prove the results of boundary-localization of transmission eigenfunctions to \eqref{eq:syst3} in $\mathbb R^3$.  Using Fourier expansion, the acoustic-elastic transmission eigenfunctions  to  \eqref{eq:syst3} associated with the transmission eigenvalues $k\in \mathbb R_+$ have series expansions as follow (cf.\cite{DR95,CDLZ21,CDLZ22,ColtonKress}):
	{	\begin{align}
			\mathbf{u}
			=\sum_{m=0}^{\infty}\sum_{l=-m}^{m}&\Big[\alpha_m^lj'_m(k_p|\mathbf{x}|)e^{\mathrm{i}m\varphi}P^{|l|}_m(\cos \theta)\cdot\hat{r}+\alpha_m^l\frac{j_m(k_p|\mathbf{x}|)}{k_p |\mathbf{x}|}e^{\mathrm{i}m\varphi}\frac{dP^{|l|}_m(\cos \theta)}{d\theta}\cdot\hat{\theta}\notag\\
			&+\alpha_m^l\frac{j_m(k_p|\mathbf{x}|)}{k_p |\mathbf{x}|}\frac{\mathrm{i}m}{\sin\theta}e^{\mathrm{i}m\varphi}P^{|l|}_m(\cos \theta)\cdot\hat{\varphi}\Big]\notag\\
			+\sum_{m=1}^{\infty}\sum_{l=-m}^{m}&\Big[\delta_m^lj_m(k_s|\mathbf{x}|)\frac{\mathrm{i}m}{\sin\theta}e^{\mathrm{i}m\varphi}P^{|l|}_m(\cos \theta)\cdot\hat{\theta}-\delta_m^lj_m(k_s|\mathbf{x}|)e^{\mathrm{i}m\varphi}\frac{dP^{|l|}_m(\cos \theta)}{d\theta}\cdot\hat{\varphi}\Big]\notag\\
			+\sum_{m=1}^{\infty}\sum_{l=-m}^{m}&\Big[\gamma_m^l \frac{j_m(k_s|\mathbf{x}|)}{k_s|\mathbf{x}|}e^{\mathrm{i}m\varphi}P^{|l|}_m(\cos \theta)\cdot\hat{r}\notag\\
			&+\gamma_m^l\Big(j'(k_s|\mathbf{x}|)+\frac{j_m(k_s|\mathbf{x}|)}{k_s|\mathbf{x}|}\Big)\Big(e^{\mathrm{i}m\varphi}\frac{dP^{|l|}_m(\cos\theta)}{d\theta}\cdot\hat{\theta}+\frac{\mathrm{i}m}{\sin\theta}e^{\mathrm{i}m\varphi}P^{|l|}_m(\cos \theta)\cdot\hat{\varphi}\Big)\Big],\notag\\
			v=\sum_{m=0}^{\infty}\sum_{l=-m}^{m}&\mathrm{i}^{m}\beta_m^l j_m(k|\mathbf{x}|)Y^l_m(\hat{r})=\sum_{m=0}^{\infty}\sum_{l=-m}^{m}\sqrt{\dfrac{2n+1}{4\pi}\dfrac{(m-|l|)!}{(m+|l|)!}}\mathrm{i}^{m}\beta_m^l j_m(k|\mathbf{x}|)e^{\mathrm{i}m\varphi}P^{|l|}_m(\cos \theta),\label{eq:3du}
	\end{align}}
	where $\alpha_m^l$,  $\beta_m^l$,  $\delta_m^l$ and $\gamma_m^l$ are complex constants, 
	\begin{equation}\label{3d:hat}
		\hat{r}:=\left(\begin{array}{c}
			\sin\theta\cos\varphi \\
			\sin\theta\sin\varphi \\
			\cos\theta
		\end{array}
		\right),\quad  \hat{\theta}:=\left(\begin{array}{c}
			\cos\theta\cos\varphi \\
			\cos\theta\sin\varphi \\
			-\sin\theta
		\end{array}
		\right), \quad \hat{\varphi}:=\left(\begin{array}{c}
			-\sin\varphi \\
			\cos\varphi \\
			0
		\end{array}
		\right),
	\end{equation}
	and $ P^{|l|}_m $ is the associated Legendre  function of degree $ m $ and order $ l $. The spherical Bessel function can be characterized by
	\begin{equation}\label{jandJ}
		j_m(|\mathbf{x}|)=\sqrt{\frac{\pi}{2|\mathbf{x}|}}J_{m+1/2}(|\mathbf{x}|). 
	\end{equation}

	Similar to Lemma  \ref{lem2.1}, in the following lemma we prove that there exists a sequence of discrete  transmission eigenvalues of \eqref{eq:syst3}, where  infinity is the  only accumulation point of this sequence. 
	\begin{lem}\label{lemma2.2}
		Consider the acoustic-elastic transmission eigenvalue  problem   \eqref{eq:syst3}. Let $ \Omega\subset \mathbb{R}^3 $ be the unit ball and $ \tau $ is chosen such that $\tau<j_{m+\frac{1}{2},1}/ j_{m+\frac{1}{2},2}<1$.  Let $ k_{m,\ell}, \ell=1,2,\ldots, $ be the transmission eigenvalues of (\ref{eq:syst3}), where $ m $ is the order of the spherical Bessel function  and $\ell $ denotes the $\ell$th eigenvalue for a fixed $m$. Then there exists a subsequence of $ \{k_{m,\ell}\} $, denoted by $ \{k_{m,s(m)}\} $, such that for $ m $
		sufficiently  large, it holds that 
		\begin{equation}\label{3d:kms}
			k_{m,s}\in \left(j_{m+\frac{1}{2},1},\ j_{m+\frac{1}{2},2}\right),
		\end{equation}	
		Furthermore, it yields that
		\begin{equation}\label{3d:con1}
			\frac{	k_{m,s}}{m+\frac{1}{2}}\to 1,\quad \mbox{as}\quad m\to+\infty.
		\end{equation}
		More specifically,
		\begin{equation}\label{3d:con2}
			k_{m,s}=\left(m+\frac{1}{2}\right)\left(1+C_0\left(m+\frac{1}{2}\right)^{-\frac{2}{3}}+o((m+\frac{1}{2})^{-\frac{2}{3}})\right),
		\end{equation} 
		where $ C_0 $ is a constant independent  of $m$.
	\end{lem}
	\begin{proof}
		Let  $ m\in\mathbb{N}_+ $ be fixed.  		
		For any nonzero constants $ \alpha_m^l $ and $ \beta_m^l $, it can be verified that
		\begin{equation}\label{3deq:u,v}
			\begin{split}
				\mathbf{u}_m(\mathbf{x})=&\alpha_m^lj'_m(k_p|\mathbf{x}|)e^{\mathrm{i}m\varphi}P^{|l|}_m(\cos \theta)\cdot\hat{r}+\alpha_m^l\frac{j_m(k_p|\mathbf{x}|)}{k_p |\mathbf{x}|}e^{\mathrm{i}m\varphi}\frac{dP^{|l|}_m(\cos \theta)}{d\theta}\cdot\hat{\theta}\\
				&+\alpha_m^l\frac{j_m(k_p|\mathbf{x}|)}{k_p |\mathbf{x}|}\frac{\mathrm{i}m}{\sin\theta}e^{\mathrm{i}m\varphi}P^{|l|}_m(\cos \theta)\cdot\hat{\varphi},\\
				v_m(\mathbf{x})=&\beta_m^l j_m(k|\mathbf{x}|)e^{\mathrm{i}m\varphi}P^{|l|}_m(\cos \theta),
			\end{split}
		\end{equation}
		are  solutions to $ \mathcal{L}_{\lambda, \mu}\bmf {u}(\mathbf{x}) + k^2\tau^2\bmf {u}(\mathbf{x})=\mathbf  0$ and $ \Delta v(\mathbf{x})+k^2v(\mathbf{x})=0$  in $ \Omega $. 
		Using the transmission conditions of (\ref{eq:syst3}) on $ \partial\Omega $, we have 
		\begin{equation}\label{3d:bound1}
			\begin{split}
				\bmf {u}_m(\mathbf{x})\cdot \nu-\frac{1}{k^2}\nabla v_m(\mathbf{x})\cdot \nu &=\bmf {u}_m\cdot\hat{r}-\frac{1}{k^2}\frac{\partial v_m}{\partial r}\\
				& = \alpha_m^l   j'_m(k_pr)e^{\mathrm{i}m\theta}P^{|l|}_m-\frac{1}{k}\beta_m^l  j'_m(kr)e^{\mathrm{i}m\theta}P^{|l|}_m=0,
			\end{split}
		\end{equation}
		where $ r=|\mathbf{x}| $. Since $ r=1 $, it yields that
		\begin{equation}\label{3d:betam}
			\beta_m^l=\alpha_m^l  k\frac{j'_m(k_p)}{j'_m(k)}.
		\end{equation}
		it is obvious that $\alpha_m^l$ and  $\beta_m^l$ are not zero by contradiction. From direct calculations, we can obtain that 
		\begin{align}
			\nabla\cdot\bmf{u}_m
			=&\alpha_m^lk_p j''_m(k_p r)e^{\mathrm{i}m\varphi}P^{|l|}_m+2\alpha_m^l\frac{j'_m( r)}{k_p r}e^{\mathrm{i}m\varphi}P^{|l|}_m+\alpha_m^l\frac{j_m(k_p r)}{k_p r^2}e^{\mathrm{i}m\varphi}\frac{d^2 P^{|l|}_m}{d\theta^2}\notag \\
			&+\alpha_m^l\frac{j_m(k_p r)}{k_p r^2}\frac{\cos\theta}{\sin\theta}e^{\mathrm{i}m\varphi}\frac{dP^{|l|}_m}{d\theta}-\alpha_m^l\frac{j_m(k_p r)}{k_p r^2}\frac{m^2}{\sin^2\theta}e^{\mathrm{i}m\varphi} P^{|l|}_m , \label{3d:grad u}
			\\
			\nabla^s\bmf{u}_m\cdot\hat{r}=&\frac{1}{2}\left(\nabla\bmf{u}_m+\nabla\bmf{u}_m^t\right)\cdot\hat{r}=\alpha_m^lk_p j''_m(k_p r)e^{\mathrm{i}m\varphi}P^{|l|}_m\cdot\hat{r} \label{3d:grad2} \\
			&+\Big(\alpha_m^l\frac{j'_m(k_p r)}{k_p r}e^{\mathrm{i}m\varphi}\frac{dP^{|l|}_m}{d\theta}-\alpha_m^l\frac{j_m(k_p r)}{k_p r^2}e^{\mathrm{i}m\varphi}\frac{dP^{|l|}_m}{d\theta}\Big)\cdot\hat \theta \notag \\
			&+\Big(\alpha_m^l\frac{\mathrm{i}m}{r\sin\theta}j'_m{k_p r}e^{\mathrm{i}m\varphi}P^{|l|}_m-\alpha_m^l\frac{j_m(k_p r)}{k_p r^2}\frac{\mathrm{i}m}{r\sin\theta}e^{\mathrm{i}m\varphi}P^{|l|}_m\Big)\cdot\hat{\varphi}. \notag
		\end{align}		
		

		By using the transmission conditions of (\ref{eq:syst3}) on $ \partial\Omega $, and combining with (\ref{3deq:u,v}), (\ref{3d:grad u}) and (\ref{3d:grad2}), we can deduce that 
		\begin{align}
			&\left(\lambda(\nabla\cdot\bmf{u}_m)I_3+2\mu	\nabla^s\bmf{u}_m+\delta\tau^2 v_m I_3\right)\cdot\hat{r}\notag \\
			=&\lambda\Big[\alpha_m^lk_p j''_m(k_p )e^{\mathrm{i}m\varphi}P^{|l|}_m+2\alpha_m^lj'_m(k_p )e^{\mathrm{i}m\varphi}P^{|l|}_m+\alpha_m^l\frac{j_m(k_p )}{k_p }e^{\mathrm{i}m\varphi}\frac{d^2 P^{|l|}_m}{d\theta^2}+\alpha_m^l\frac{j_m(k_p )}{k_p}\frac{\cos\theta}{\sin\theta}e^{\mathrm{i}m\varphi}\notag \\
			&\times \frac{dP^{|l|}_m}{d\theta}-\alpha_m^l\frac{j_m(k_p )}{k_p }\frac{m^2}{\sin^2\theta}e^{\mathrm{i}m\varphi} P^{|l|}_m \Big]\cdot\hat{r}+2\mu\alpha_m^lk_p j''_m(k_p r)e^{\mathrm{i}m\varphi}P^{|l|}_m\cdot\hat{r}+2\mu\Big[\alpha_m^l  j'_m(k_p )e^{\mathrm{i}m\varphi} \notag \\
			&\times \frac{dP^{|l|}_m}{d \theta}-\alpha_m^l\frac{j_m(k_p )}{k_p }e^{\mathrm{i}m\varphi}\frac{dP^{|l|}_m}{d\theta}\Big]\cdot\hat{\theta}+2\mu\Big[\alpha_m^l\frac{\mathrm{i}m}{\sin\theta}j'_m({k_p })e^{\mathrm{i}m\varphi}P^{|l|}_m-\alpha_m^l\frac{j_m(k_p )}{k_p }\frac{\mathrm{i}m}{ \sin\theta}e^{\mathrm{i}m\varphi} \notag \\
			&\times P^{|l|}_m\Big]\cdot\hat{\varphi}+\delta\tau^2\beta_m^lj_m(k)e^{\mathrm{i}m\varphi}P^{|l|}_m\cdot\hat{r}=\mathbf 0. \label{3d:bound2}
		\end{align}
		Using the linear independence of $\hat r$, $ \hat{\theta} $ and $ \hat{\varphi} $, in  view of the coefficient of $ \hat{\varphi} $ and (\ref{3d:betam}),  we obtain that		
		\begin{equation}\label{3d:rel1}
			j_m(k_p)=k_pj'_m(k_p).
		\end{equation} 
		Similarly,  according to the coefficient  of $ \hat{r} $ in \eqref{3d:bound2}, by using (\ref{eq:k_p}), (\ref{jandJ}),(\ref{3d:betam}), (\ref{3d:rel1}) and the following property (cf.\cite{Abr})		\begin{subequations}
			\begin{align}
				j'_m(|\mathbf{x}|)+j_{m+1}(|\mathbf{x}|)-\frac{m}{|\mathbf{x}|}j_{m}(|\mathbf{x}|)&=0,\label{3d:func1}\\
				\label{3d:func2}
				|\mathbf{x}|^2j''_m(|\mathbf{x}|)+2|\mathbf{x}|j'_m(|\mathbf{x}|)+\left[|\mathbf{x}|^2-m(m+1)\right]j_m(|\mathbf{x}|)&=0,\\
				\label{3d:func3} \frac{d^2P_m^{|l|}}{d\theta^2}+\frac{\cos\theta}{\sin\theta}\frac{dP_m^{|l|}}{\theta}+\Big[m(m+1)-\frac{m^2}{\sin^2\theta}\Big]P_m^{|l|}&=0,
			\end{align}
		\end{subequations}
		after some calculations, we have 
		{\begin{equation}\label{3d:rel1.5}
				-4\mu k_pj'_m(k_p)j'_m(k)-k_p^2j_m(k_p)j'_m(k)+2\mu m(m+1)j_m(k_p)j'_m(k)+\delta\tau^2kj'_m(k_p)j_m(k)=0
			\end{equation} Now, we want to find a $k$ that satisfies both (\ref{3d:rel1}) and (\ref{3d:rel1.5}). Then, we substitute (\ref{3d:rel1}) into (\ref{3d:rel1.5}), one has 
			\begin{equation}\label{3d:rel2}
				\left[-4\mu-k_p^2+2\mu m(m+1) \right]j_m(k_p)j'_m(k)+\frac{\delta\tau^2 k}{k_p}j_m(k_p)j_m(k)=0
		\end{equation}
		Apparently, if $k$ solves the equations \eqref{3d:rel1} and \eqref{3d:rel1.5}, the equation \eqref{3d:rel2} holds. On the other hand, if $k$ is the root of the equations \eqref{3d:rel1} and \eqref{3d:rel2}, $k$  solves the equation \eqref{3d:rel1.5}. Thus, in the rest of the proof, we shall confine ourselves on finding the root $k$ of the equation \eqref{3d:rel2}. Then we will show that the root $k$ solves \eqref{3d:rel1} in the asymptotic sense and please refer to Remark \ref{3drem:as18}.}
	
		By using (\ref{jandJ}), we find that 
		the transmission eigenvalues $ k $'s to \eqref{eq:syst3} are positive zeros of the following function:
		\begin{equation}
			\begin{split}	\label{3d:rela1}
				f_{m+\frac{1}{2}}(k):=&\Big[\frac{4\mu}{k\sqrt{\tau}}+\tau^{\frac{3}{2}}-\frac{2\mu m(m+1)}{\sqrt{\tau}k}\Big]J_{m+\frac{3}{2}}(k)J_{m+\frac{1}{2}}(\tau k)\\
				&+\Big[-\frac{4\mu m}{k^2\sqrt{\tau}}-m\tau^{\frac{3}{2}}-\frac{2\mu m^2(m+1)}{\sqrt{\tau}k^2}+\delta\tau^{\frac{3}{2}}\Big]J_{m+\frac{1}{2}}(k)J_{m+\frac{1}{2}}(\tau k).
			\end{split}
		\end{equation}
		Nest, we compute the following quantity:
		\begin{equation}\begin{split}\label{3df}
				&f_{m+\frac{1}{2}}(j_{m+\frac{1}{2},1})f_{m+\frac{1}{2}}(j_{m+\frac{1}{2},2})\\
				=&\Big[\frac{4\mu}{j_{m+\frac{1}{2},1}\sqrt{\tau}}+\tau^{\frac{3}{2}}-\frac{2\mu m(m+1)}{\sqrt{\tau}j_{m+\frac{1}{2},1}}\Big]\Big[\frac{4\mu}{j_{m+\frac{1}{2},2}\sqrt{\tau}}+\tau^{\frac{3}{2}}-\frac{2\mu m(m+1)}{\sqrt{\tau}j_{m+\frac{1}{2},2}}\Big]\\
				&\times J_{m+\frac{1}{2}}(\tau j_{m+\frac{1}{2},1})J_{m+\frac{1}{2}}(\tau j_{m+\frac{1}{2},2})J_{m+\frac{3}{2}}(j_{m+\frac{1}{2},1})J_{m+\frac{3}{2}}(j_{m+\frac{1}{2},2})\\
			\end{split}
		\end{equation}	
   Nest, using (\ref{eq:jms up}) when $ m $ is sufficiently large, there hold that
	\begin{equation}\begin{split}\label{3dj_m}
			j_{m+\frac{1}{2},1}&=\left(m+\frac{1}{2}\right)\Big(1+C_1(m+\frac{1}{2})^{-2/3}+o((m+\frac{1}{2})^{-2/3})\Big),\\
			j_{m+\frac{1}{2},2}&=\left(m+\frac{1}{2}\right)\Big(1+C_2(m+\frac{1}{2})^{-2/3}+o((m+\frac{1}{2})^{-2/3})\Big),
		\end{split}
	\end{equation} where $ C_i,i=1,2 $ is a positive constant not depending on $m$.  By virtue of \eqref{3dj_m}, when $m$ is sufficient large, we have 
	{	\begin{equation}\begin{split}\label{3deq:1}
			&\Big[\frac{4\mu}{j_{m+\frac{1}{2},1}\sqrt{\tau}}+\tau^{\frac{3}{2}}-\frac{2\mu m(m+1)}{\sqrt{\tau}j_{m+\frac{1}{2},1}}\Big]\Big[\frac{4\mu}{j_{m+\frac{1}{2},2}\sqrt{\tau}}+\tau^{\frac{3}{2}}-\frac{2\mu m(m+1)}{\sqrt{\tau}j_{m+\frac{1}{2},2}}\Big]\\
			&=m\,\Pi_{i=1}^2\left[\frac{2\mu}{\sqrt{\tau}}(1-2m^{-2})(1-C_i(m+\frac{1}{2})^{-2/3}-o((m+\frac{1}{2})^{-2/3}))-\tau^{3/2} m^{-1} \right] >0
	\end{split}\end{equation}
	for any positive constants $\tau$ and $\mu$. 
	
			Moreover, from the choice of $\tau j_{m+\frac{1}{2},2}<j_{m+\frac{1}{2},1}$, the following inequality holds
			\begin{equation}
				J_{m+\frac{1}{2}}(\tau j_{m+\frac{1}{2},1})J_{m+\frac{1}{2}}(\tau j_{m+\frac{1}{2},2})>0.
			\end{equation}
			Since the positive zeros of $J_{m+\frac{3}{2}} $ are interlaced with those of $J_{m+\frac{1}{2}}$, we have 
			\[
			J_{m+\frac{3}{2}}(j_{m+\frac{1}{2},1})J_{m+\frac{3}{2}}(j_{m+\frac{1}{2},2})<0.
			\]
			By using the above fact, we derive that 
			\begin{equation}\begin{split}\label{3d:f}
				f_{m+\frac{1}{2}}(j_{m+\frac{1}{2},1})f_{m+\frac{1}{2}}(j_{m+\frac{1}{2},2})<0,
				\end{split}
			\end{equation}  which implies that $k_{m,s}\in \left(j_{m+\frac{1}{2},1},j_{m+\frac{1}{2},2}\right) $.
			
			
			
			
			
			
			Finally, by using  (\ref{3dj_m}) one has (\ref{3d:con1}) and (\ref{3d:con2}).}
		
		The proof is complete.	
	\end{proof}
		{\begin{rem}\label{3drem:as18}
			Here we remark that if $k_{m,s}$ is a root of \eqref{3d:rel2}, the root $k_{m,s}$ solves the equation \eqref{3d:rel1} in the asymptotic sense. 
			Indeed, if $k_{m,s}$ is a root of (\ref{3d:rel2}), for sufficiently large $m$ and $\tau<j_{m+\frac{1}{2},1}/j_{m+\frac{1}{2},2}<1$, using (\ref{asy1}), (\ref{jandJ}), (\ref{3d:con1}) and (\ref{3d:con2}) 
			we derive that
			\begin{equation}
					j_m(k_p)-k_pj'_m(k_p)=C\left(\frac{\tau e^{\sqrt{1-\tau^2}}}{1+\sqrt{1-\tau^2}}\right)^{m-\frac{1}{2}}\rightarrow 0 \quad \mbox{as} \quad m\rightarrow \infty,
			\end{equation} 
			where C is a constant. In the derivation of the last inequality, we have used the fact $0<\frac{\tau e^{\sqrt{1-\tau^2}}}{1+\sqrt{1-\tau^2}}<1$ due to $\tau<1$ . Thus, $k_{m,s}$ satisfies (\ref{3d:rel1}) in the asymptotic sense. 
		\end{rem}}

		
			

	We are in a position to consider the boundary-localization patterns of the AE transmission eigenfunctions associated with the eigenvalues determined in Lemma \ref{lemma2.2}.
	

	\begin{thm}\label{3d:th1}
		Consider (\ref{eq:syst3}) and let $ \Omega $ be the unit ball in $ \mathbb{R}^3 $ and $ \Omega_{\epsilon} $ be given in (\ref{eq:Omega_tau}). For any $ \epsilon\in(0,1) $, there exist transmission  eigenfunctions $ v_m $, $m\in\mathbb{N}$, associated with the transmission eigenvalues $ k_{m,s} $ described in Lemma \ref{lemma2.2}  such that 
		\begin{equation}\label{eq:3dth:sur1}
			\lim_{ m \to \infty}\frac{\|v_m\|_{L^2(\Omega_{\epsilon})}}{\|v_m\|_{L^2(\Omega)}}=0.
		\end{equation}
	\end{thm}
	\begin{proof}
		Recall that the transmission eigenvalue  $k_{m,s}$ fulfils  \eqref{3d:con2}. Let $v_m$ be given by (\ref{3deq:u,v}).  One has
		\begin{subequations}
			\begin{align}
				\label{3d:v_m1}
				\|v_m\|^2_{L^2(\Omega_{\epsilon})} &=|\beta_m^l|^2\int_{\Omega_{\epsilon}}|j_m(k_{m,s}|\mathbf{x}|)|^2 d\mathbf{x}=\frac{2\pi|\beta_m^l|^2}{k_{m,s}}\int_{0}^{\epsilon}rJ^2_{m+\frac{1}{2}}(k_{m,s}r) dr,\\
				\label{3d:v_m2}	\|v_m\|^2_{L^2(\Omega)}&=|\beta_m^l|^2\int_{\Omega}|j_m(k_{m,s}|\mathbf{x}|)|^2 d\mathbf{x}=\frac{2\pi|\beta_m^l|^2}{k_{m,s}}\int_{0}^{1}rJ^2_{m+\frac{1}{2}}(k_{m,s}r) dr.
			\end{align}
		\end{subequations}
		
		From the proof of Theorem \ref{th2.4}, we have $ J'_{m+\frac{1}{2}}(k_{m,s}r) >0 $ for $ 0<r<1 $. Hence, by using (\ref{asy1}),we can derive the following asymptotic bound:
		\begin{equation}
			\begin{split}\label{3d:vmeps1}
				&	\int_{0}^{\epsilon}rJ^2_{m+\frac{1}{2}}(k_{m,s}r) dr
				\leqslant \epsilon^2J^2_{m+\frac{1}{2}}(k_{m,s}\epsilon)\\
				=&\frac{\epsilon^2}{2\pi (m+\frac{1}{2})\left(1-(\frac{k_{m,s}\epsilon}{m+\frac{1}{2}})^2\right)^{1/2}}\left(\frac{k_{m,s}\epsilon}{m+\frac{1}{2}} \frac{e^{\sqrt{1-(\frac{k_{m,s}\epsilon}{m+\frac{1}{2}})^2}}}{1+\sqrt{1-(\frac{k_{m,s}\epsilon}{m+\frac{1}{2}})^2}} \right)^{2(m+\frac{1}{2})} \left(1+o(1)\right)^2
			\end{split}
		\end{equation}   
		as $m\rightarrow \infty$.  For  sufficiently large $ m $, one can choose $ \Xi_2 $ by
		\begin{equation}\label{3d:r2}
			\Xi_2=\frac{j'_{m+\frac{1}{2},1}}{k_{m,s}},
		\end{equation} 
		where $ \epsilon<\Xi_2<1 $. From \eqref{eq:jpm1}, we have $ j'_{m+\frac{1}{2},1}=(m+\frac{1}{2})\left(1+{O}((m+\frac{1}{2})^{-2/3})\right) $. Hence, one has $m+\frac{1}{2}<j'_{m+\frac{1}{2},1}=k_{m,s}\Xi_2 $. It yields that
		{	 \begin{align}\label{3d:vm1}
				\int_{0}^{1}rJ^2_{m+\frac{1}{2}}(k_{m,s}r) dr
				\geqslant&\int_{\Xi_2}^{1}rJ^2_{m+\frac{1}{2}}(k_{m,s}r) dr
				=\frac{2(1+o(1))^2}{\pi}\int_{\Xi_2}^{1}\frac{r\cos^2 x_{m+\frac{1}{2}}}{\left( k^2_{m,s}r^2-(m			+\frac{1}{2})^2\right)^{1/2}}dr\notag\\
				\sim&\frac{1}{\pi}\int_{\Xi_2}^{1}\frac{r}{\left( k^2_{m,s}r^2-(m+\frac{1}{2})^2\right)^{1/2}} dr\notag\\
				\sim&\frac{1}{\pi k^2_{m,s}}\Big(\sqrt{ k_{m,s}^2-(m+\frac{1}{2})^2}-\sqrt{j_{m+\frac{1}{2},1}'^2-(m+\frac{1}{2})^2}\Big).
			\end{align}
		where  $x_{m+\frac{1}{2}}=\sqrt{{(\tau k_{m,s}r)}^2-(m+\frac{1}{2})^2}-\frac{(m+\frac{1}{2})\pi}{2}+(m+\frac{1}{2})\arcsin(\frac{m+\frac{1}{2}}{\tau k_{m,s}r})-\frac{\pi}{4} $. 	
		
		Therefore, combining (\ref{3d:v_m1}), (\ref{3d:v_m2}), (\ref{3d:vmeps1}) and (\ref{3d:vm1}), there holds
		\begin{equation}\begin{split}\label{3d:compare1}
				\frac{\|v_m\|^2_{L^2(\Omega_{\epsilon})}}{\|v_m\|^2_{L^2(\Omega)}}\leqslant&C(m+\frac{1}{2})^{\frac{1}{3}}\left(1-\epsilon^2\right)^{-\frac{1}{2}}\left(\frac{\epsilon e^{\sqrt{1-\epsilon^2}}}{1+\sqrt{1-\epsilon^2}} \right)^{2(m+\frac{1}{2})} 
		\end{split}\end{equation} 
		where $C$ is a constant. 
		By letting $ m\to\infty $ one has (\ref{eq:3dth:sur1})}
\end{proof}

\begin{thm}\label{3d:th2}
	Consider (\ref{eq:syst3}) and let $ \Omega $ be the unit ball in $ \mathbb{R}^3 $ and $ \Omega_{\epsilon} $ be given in (\ref{eq:Omega_tau}). Let the physical parameter $ \tau $ in \eqref{eq:syst3} satisfy $ \tau \in (0, 1) $. For any $ \epsilon\in(0,1) $, there exists transmission eigenfunctions $ \mathbf{u}_m $, $m\in\mathbb{N}$, associated with the transmission eigenvalues $ k_{m,s} $ described in Lemma \ref{lemma2.2} such that 
	\begin{equation}\label{eq:3dth:sur4}
		\lim_{ m \to \infty}\frac{\|\mathbf{u}_m\|_{L^2(\Omega_{\epsilon})}}{\|\mathbf{u}_m\|_{L^2(\Omega)}}=0.
	\end{equation}
\end{thm}
\begin{proof}
	Suppose that the transmission eigenvalue  $k_{m,s}$ fulfils  \eqref{3d:con2}. Let the transmission eigenfunction $\mathbf u_m$ of \eqref{eq:syst3} be associated with  $k_{m,s}$ in (\ref{3deq:u,v}).  By using (\ref{jandJ}), (\ref{3d:func1}) and the following properties (cf. \cite{Abr}):
	\begin{align}
		\label{P1}
		\frac{dP_m^{|l|}(\cos\theta)}{d\theta}&=\frac{1}{2}\left[(m+|l|)(m-|l|+1)P_{m}^{|l|-1}(\cos\theta)-P_{m}^{|l|+1}(\cos\theta)\right],\\
		\label{P2}
		\frac{|m|}{\sin\theta}P_m^{|l|}(\cos\theta){d\theta}&=-\frac{1}{2}\left[P_{m-1}^{|l|+1}(\cos\theta)+(m+|l|-1)(m+|l|)P_{m-1}^{|l|-1}(\cos\theta)\right],\\
		\label{P3}
		\int_{-1}^{1}P_m^{l}(x)P_m^{k}(x)&=\begin{cases}
			0, &l\ne k\\
			\frac{(m+|l|)!}{(m-|l|)!}\frac{2}{2m+1},&l=k
		\end{cases},
	\end{align}
	and similar to \eqref{3d:vmeps1}, one can show that 
	\begin{align}\label{3d:u}
		\|\bmf{u}_m\|^2_{L^2(\Omega_\epsilon)}
		&=\int_{\Omega_\epsilon}\Bigg( |\alpha_m^l|^2|j'_m(k_p|\bmf{x}|)|^2|P_m^{|l|}|^2+|\alpha_m^l|^2\frac{|j_m(k_p|\bmf{x}|)|^2}{k_p^2|\bmf{x}|^2}(\frac{dP_m^{|l|}}{d\theta})^2\notag \\
		&\quad +|\alpha_m^l|^2\frac{|j_m(k_p|\bmf{x}|)|^2}{k_p^2|\bmf{x}|^2}\frac{m^2}{\sin^2\theta}|P_m^{|l|}|^2 \Bigg) d\bmf{x}\notag \\
		\leqslant&\frac{|\alpha_m^l|^2}{k_{m,s}}\frac{(m+|l|)!}{(m-|l|)!(m+\frac{1}{2})}\frac{\epsilon}{2\pi(m+\frac{1}{2})\sqrt{1-(\frac{\tau k_{m,s}\eta_3}{m+\frac{1}{2}})^2}}\\
		&\times \Biggl[\frac{\eta_3}{\tau }\Biggr(\frac{\tau k_{m,s}\eta_3}{m+\frac{1}{2}}\frac{e^{\sqrt{1-(\frac{\tau k_{m,s}\eta_3}{m+\frac{1}{2}})^2}}}{1+\sqrt{1-(\frac{\tau k_{m,s}\eta_3}{m+\frac{1}{2}}})^2}\Biggr)^3\notag \\
		&+\frac{Cm^2}{\tau^3 k_{m,s}^2}\Biggr(\frac{\tau k_{m,s}}{m+\frac{1}{2}}\frac{e^{\sqrt{1-(\frac{\tau k_{m,s}\eta_3}{m+\frac{1}{2}})^2}}}{1+\sqrt{1-(\frac{\tau k_{m,s}\eta_3}{m+\frac{1}{2}}})^2}\Biggr)+\frac{Cm^2\eta_3}{\tau^3 k_{m,s}^2}  \notag \\
		&\times  \Biggr(\frac{\tau k_{m,s}\eta_3}{m+\frac{1}{2}}\frac{e^{\sqrt{1-(\frac{\tau k_{m,s}\eta_3}{m+\frac{1}{2}})^2}}}{1+\sqrt{1-(\frac{\tau k_{m,s}\eta_3}{m+\frac{1}{2}}})^2}\Biggr) \Biggr]  \Biggr(\frac{\tau k_{m,s}\eta_3}{m+\frac{1}{2}}\frac{e^{\sqrt{1-(\frac{\tau k_{m,s}\eta_3}{m+\frac{1}{2}})^2}}}{1+\sqrt{1-(\frac{\tau k_{m,s}\eta_3}{m+\frac{1}{2}}})^2}\Biggr)^{2m}(1+o(1))^2 ,\notag
	\end{align}
	where $ \eta_3\in[0,\epsilon] $. 
	
	
	When $\tau < 1$. Recall that $\Xi_2$ is given in \eqref{3d:r2}.  Adopting a similar argument for deriving \eqref{3d:vm1},  one has 
	\begin{align}
		\label{3d:u_eps}
		\|\bmf{u}_m\|^2_{L^2(\Omega)}&\geqslant\frac{|\alpha_m^l|^2}{k_{m,s}}\frac{(m+|l|)!}{(m-|l|)!(m+\frac{1}{2})}\frac{1-\Xi_2}{2\pi(m+\frac{1}{2})\sqrt{1-(\frac{\tau k_{m,s}\eta_4}{m+\frac{1}{2}})^2}} \\
		&\times  \Biggl[\frac{\eta_4}{\tau }\Biggr(\frac{\tau k_{m,s}\eta_4}{m+\frac{1}{2}}\frac{e^{\sqrt{1-(\frac{\tau k_{m,s}\eta_4}{m+\frac{1}{2}})^2}}}{1+\sqrt{1-(\frac{\tau k_{m,s}\eta_4}{m+\frac{1}{2}}})^2}\Biggr)^3 +\frac{Cm^2}{\tau^3 k_{m,s}^2}\Biggr(\frac{\tau k_{m,s}}{m+\frac{1}{2}}\frac{e^{\sqrt{1-(\frac{\tau k_{m,s}\eta_4}{m+\frac{1}{2}})^2}}}{1+\sqrt{1-(\frac{\tau k_{m,s}\eta_4}{m+\frac{1}{2}}})^2}\Biggr) \notag \\
		& +\frac{Cm^2\eta_4}{\tau^3 k_{m,s}^2}\Biggr(\frac{\tau k_{m,s}\eta_4}{m+\frac{1}{2}}\frac{e^{\sqrt{1-(\frac{\tau k_{m,s}\eta_4}{m+\frac{1}{2}})^2}}}{1+\sqrt{1-(\frac{\tau k_{m,s}\eta_4}{m+\frac{1}{2}}})^2}\Biggr)-\frac{2m}{\tau^2 k^2_{m,s}}\notag  \\
		& \times \Biggr(\frac{\tau k_{m,s}\eta_4}{m+\frac{1}{2}}\frac{e^{\sqrt{1-(\frac{\tau k_{m,s}\eta_4}{m+\frac{1}{2}})^2}}}{1+\sqrt{1-(\frac{\tau k_{m,s}\eta_4}{m+\frac{1}{2}}})^2}\Biggr)^2\Biggr]   \Biggr(\frac{\tau k_{m,s}\eta_4}{m+\frac{1}{2}}\frac{e^{\sqrt{1-(\frac{\tau k_{m,s}\eta_4}{m+\frac{1}{2}})^2}}}{1+\sqrt{1-(\frac{\tau k_{m,s}\eta_4}{m+\frac{1}{2}}})^2}\Biggr)^{2m}(1+o(1))^2 , \notag	
	\end{align}
	where $ \eta_4\in[\Xi_2,1] $ and $ \eta_4\to 1 $ as $m\to\infty  $.
	Furthermore, by using a similar  argument for \eqref{compare2},  there exists $ 0<\delta(\eta_3) <1$ such that 
	\begin{equation}\label{3d:compare2}
		\frac{\|\mathbf{u}_m\|_{L^2(\Omega_\epsilon)}}{\|\mathbf{u}_m\|_{L^2(\Omega)}}\leqslant C(m+\frac{1}{2})^{\frac{2}{3}}(1-\delta(\eta_3))^{2m},
	\end{equation}
	where $C$ is positive constant independent of $m$. Therefore, we can   prove (\ref{eq:3dth:sur4}) by letting $m\rightarrow \infty$ in \eqref{3d:compare2}.
	The proof is complete.
\end{proof}

\subsection{Numerical results}

We have conducted extensive numerical experiments to verify that the boundary-localization phenomenon of the AE transmission eigenfunctions holds for general domains of $\Omega$: radial/non-radial, smooth/non-smooth and convex/non-convex. {\color{blue} Here, we apply the finite element method for the numerical experiments. The system \eqref{eq:syst2} is first reformulated as a variational problem and the boundary conditions are the third and the forth conditions in \eqref{eq:syst2}. The mesh is chosen to be the triangular and the piecewise linear basis functions are used in the computations.} In what follows, we present a few representative examples for illustration; see Fig.~\ref{fig:1} and \ref{fig:2}.

\begin{figure}[htbp]
	\subfigure[]{
		\includegraphics[width=0.22\textwidth]{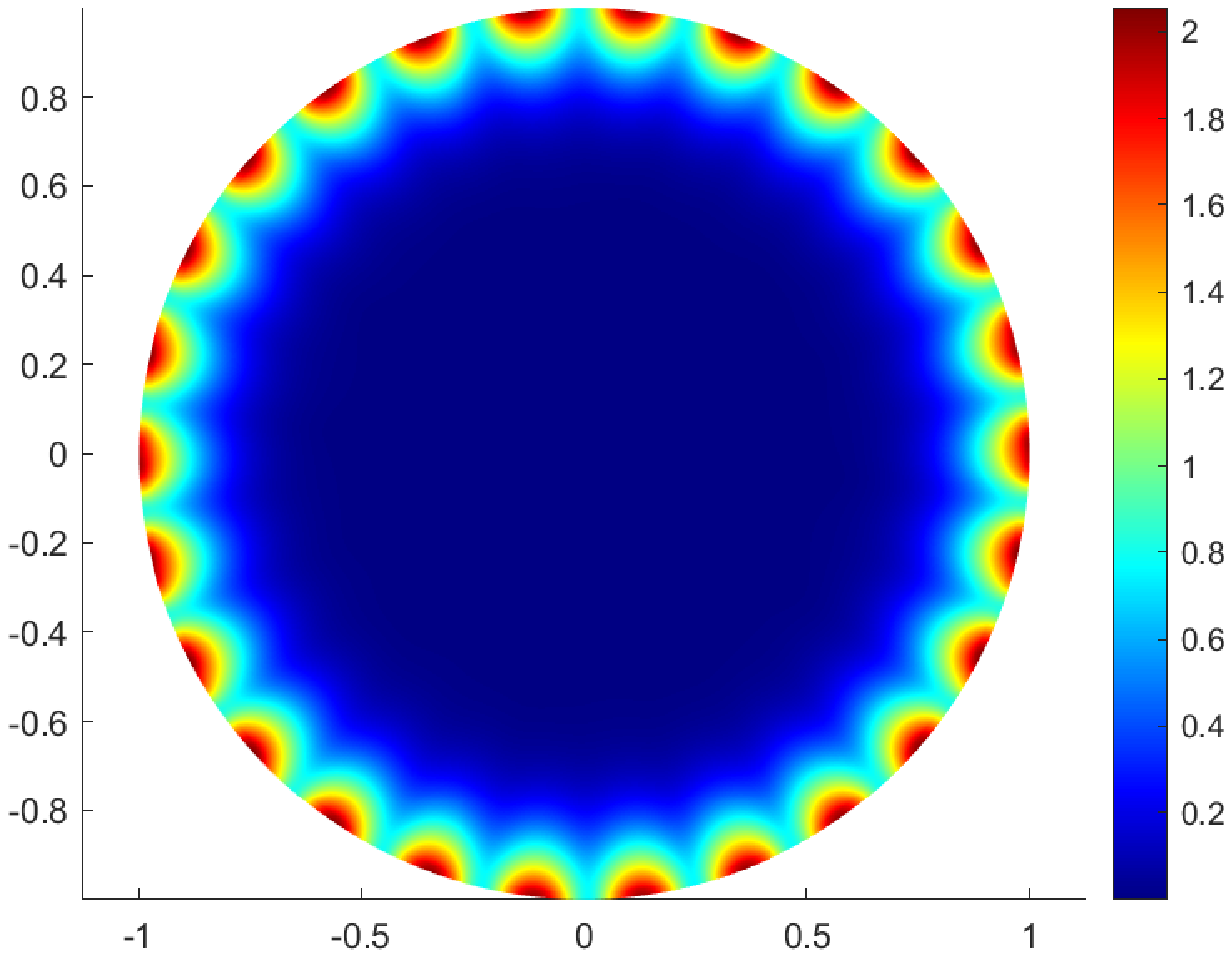}}
	\hspace{.1cm}
	\subfigure[]{
		\includegraphics[width=0.22\textwidth]{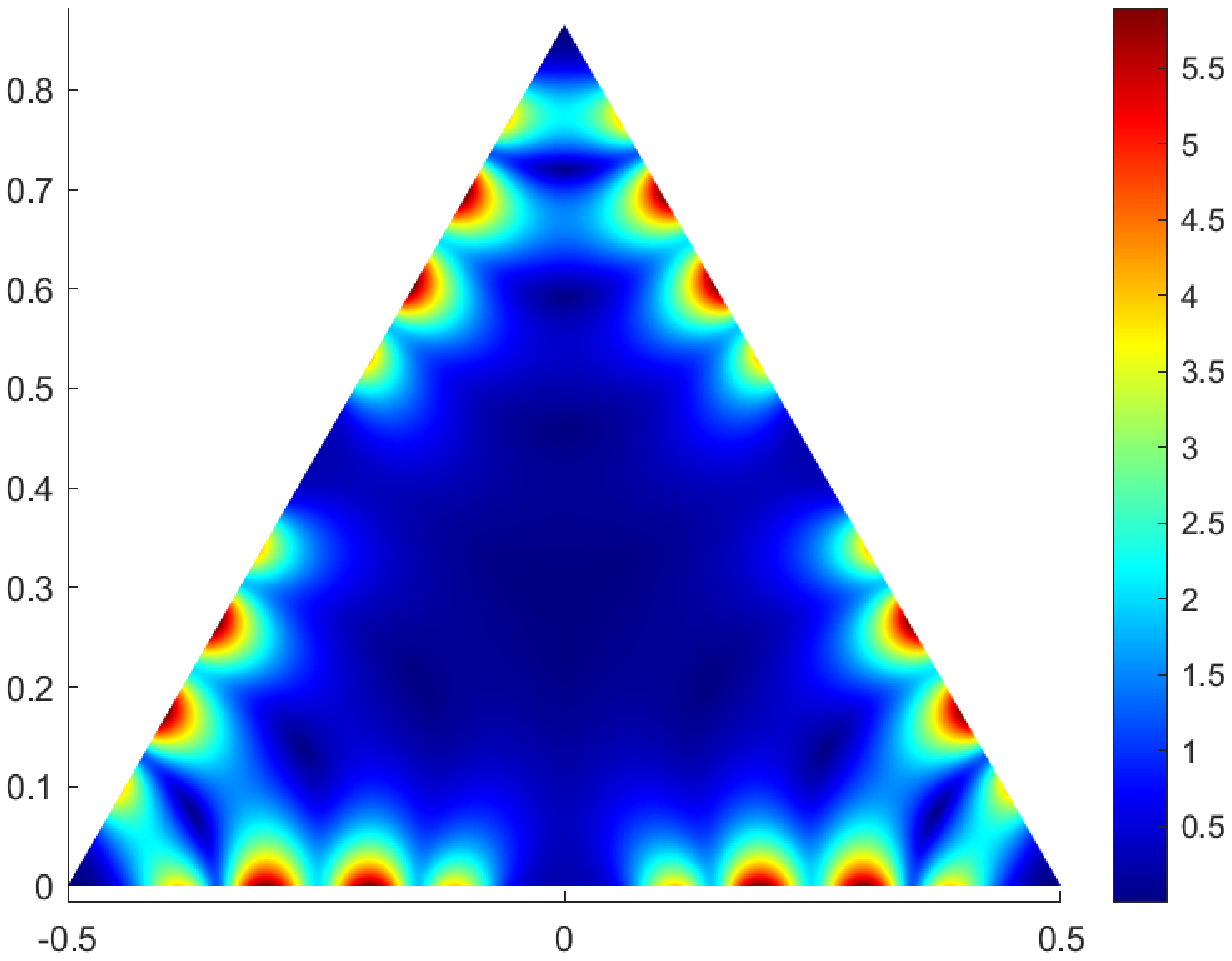}}
	\hspace{.1cm}
	\subfigure[]{
		\includegraphics[width=0.22\textwidth]{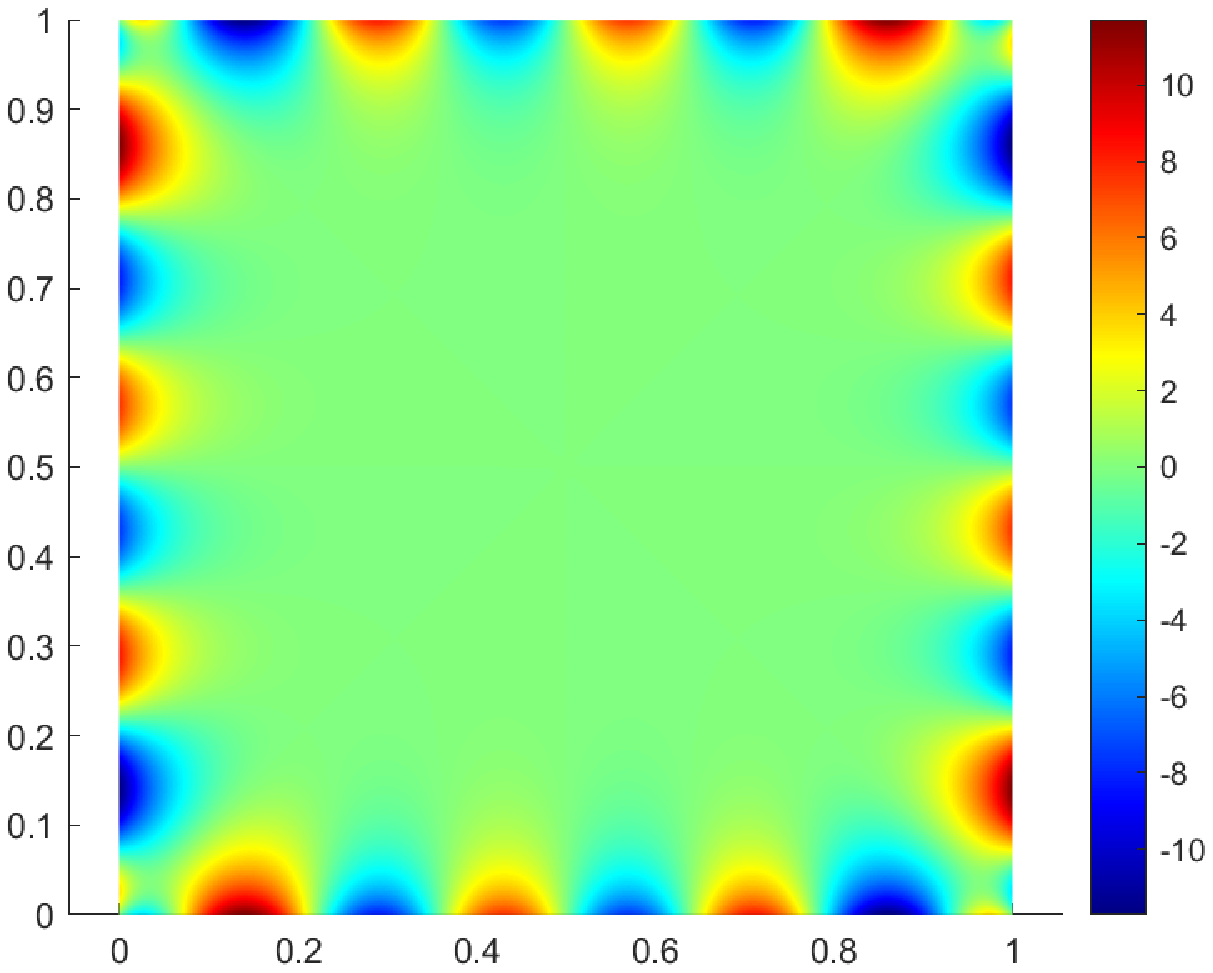}}
	\hspace{.1cm}
	\subfigure[]{
		\includegraphics[width=0.22\textwidth]{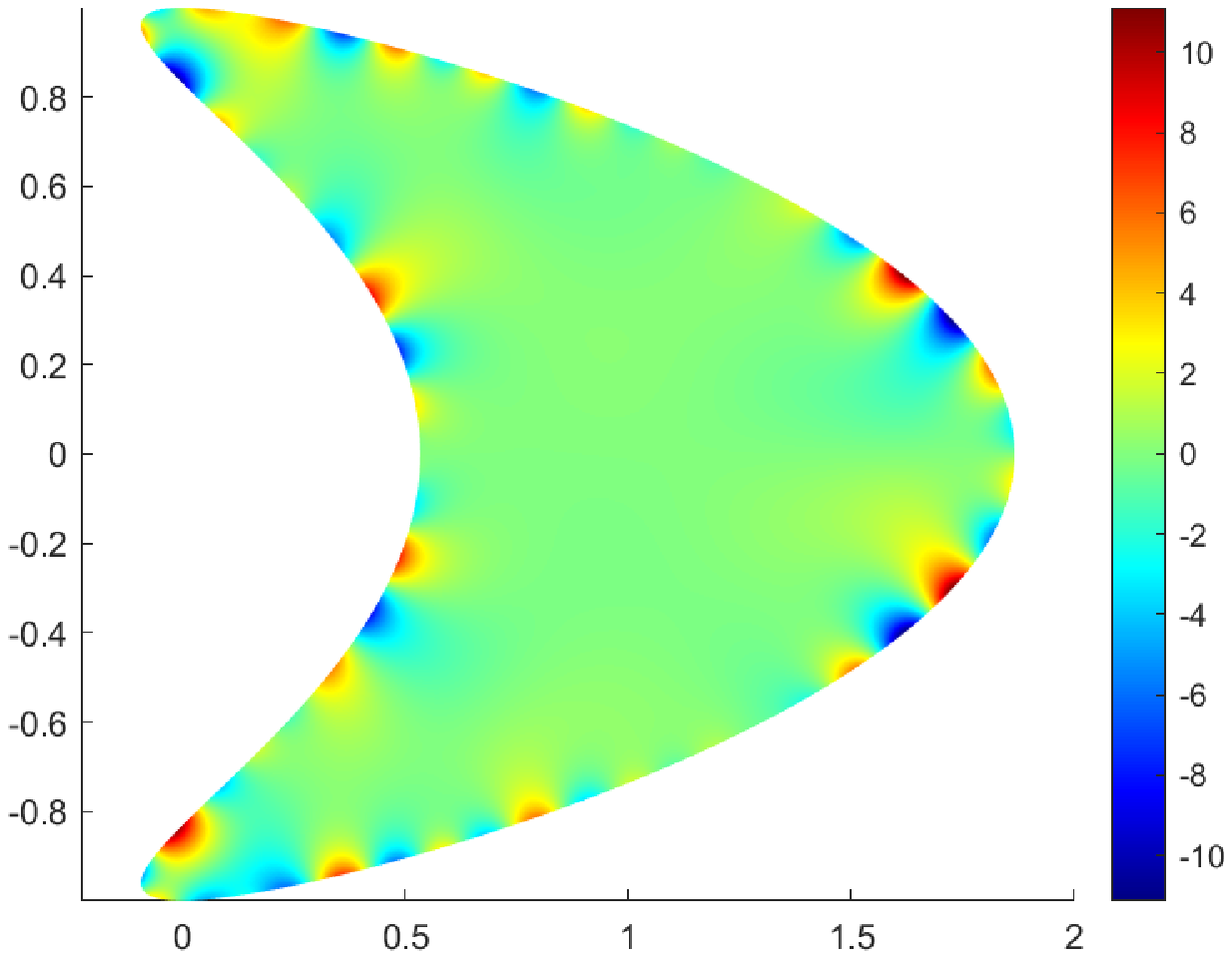}}
	\caption{(a) \& (b): $|\mathbf{u}|$ to \eqref{eq:syst2} associated with $\rho_b=20$, $\kappa=\rho_e=\tilde\lambda=\tilde\mu=1$ and $w^2=43.27, 201.52$, respectively; (c) \& (d): $v$ to \eqref{eq:syst2} associated with $\rho_e=10$, $\kappa=\rho_b=\tilde\lambda=\tilde\mu=1$ and $w^2=93.71, 101.35$, respectively. }
	\label{fig:1}
\end{figure}
\begin{figure}[htbp]
	\subfigure[]{
		\includegraphics[width=0.22\textwidth]{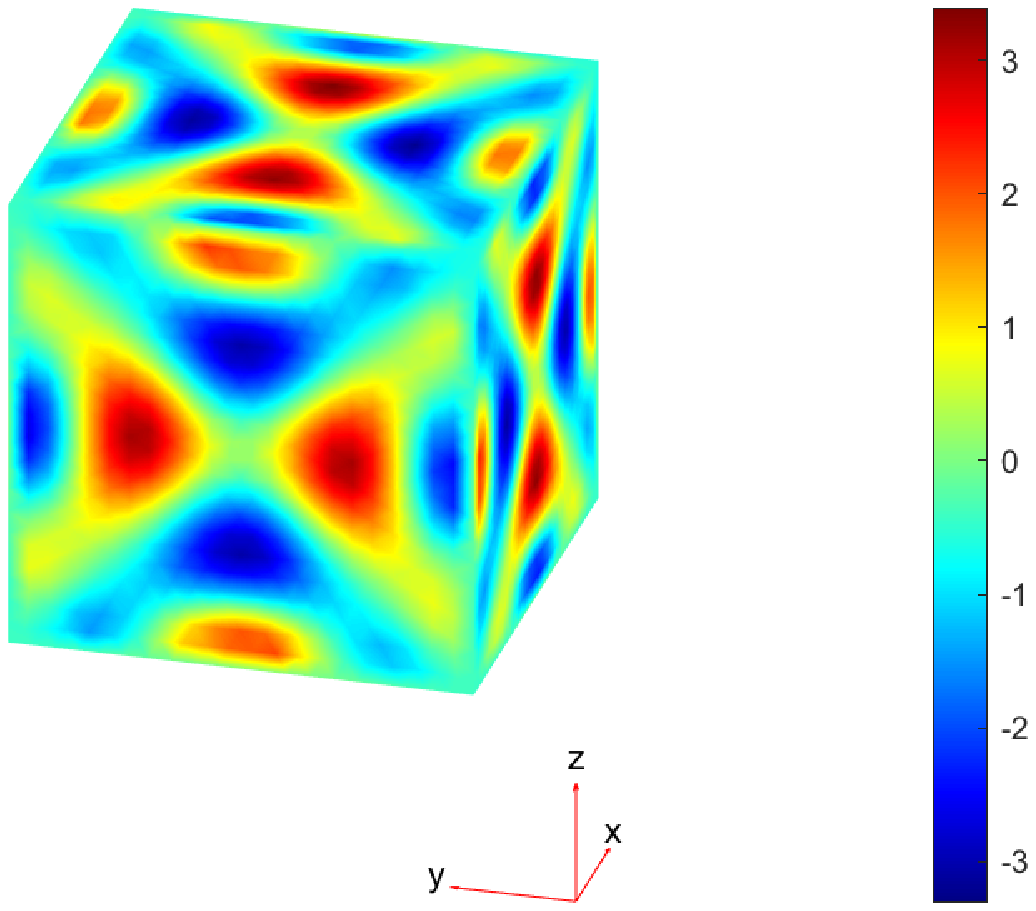}
		\includegraphics[width=0.22\textwidth]{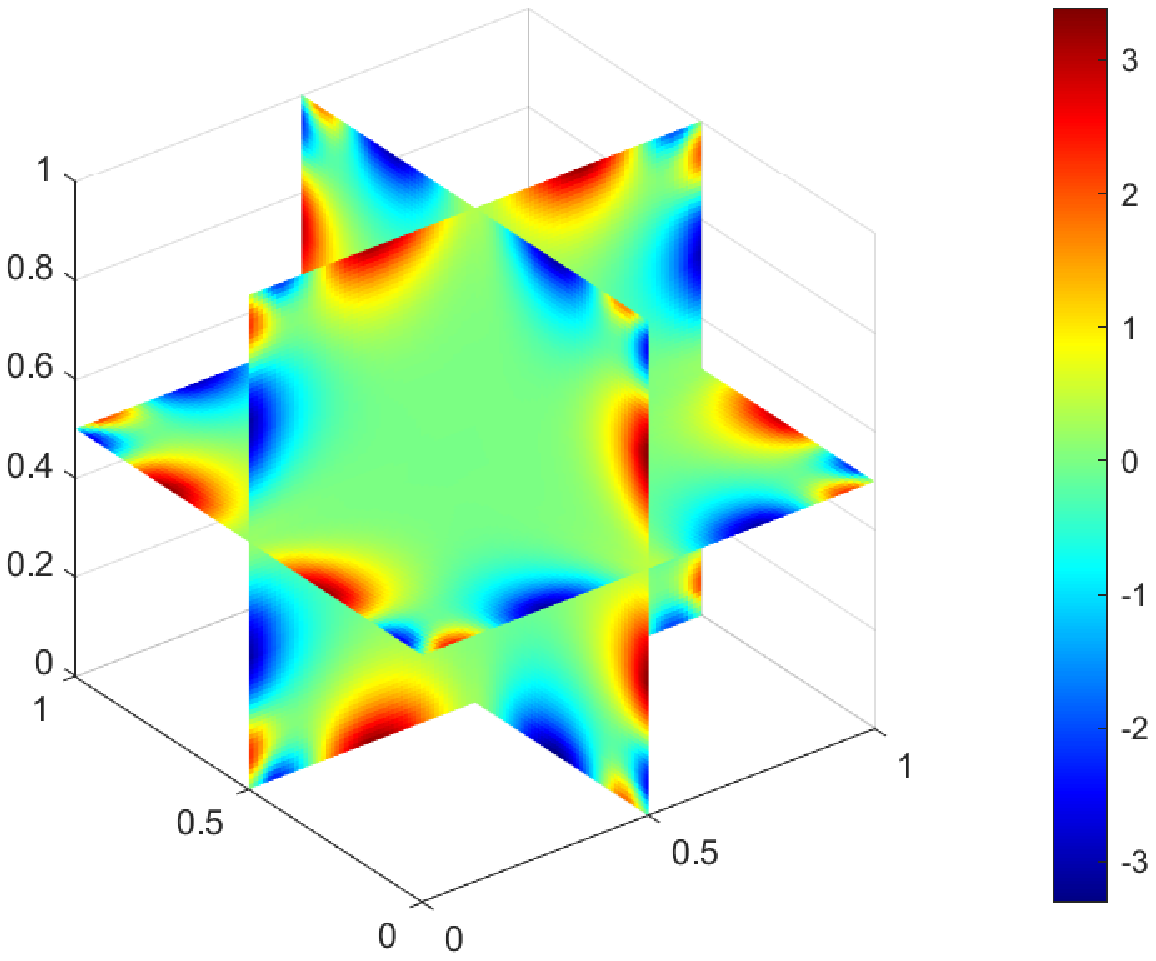}
	}
	\hspace{.1cm}
	\subfigure[]{
		\includegraphics[width=0.22\textwidth]{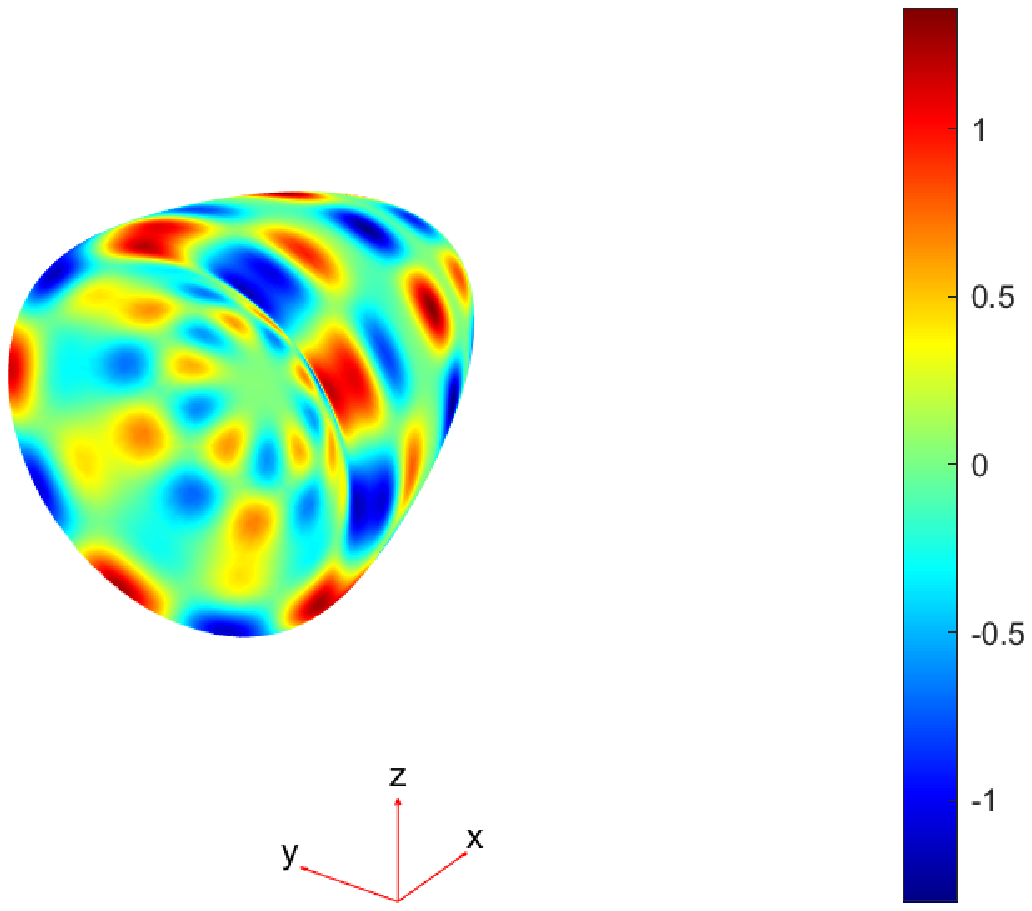}
		\includegraphics[width=0.22\textwidth]{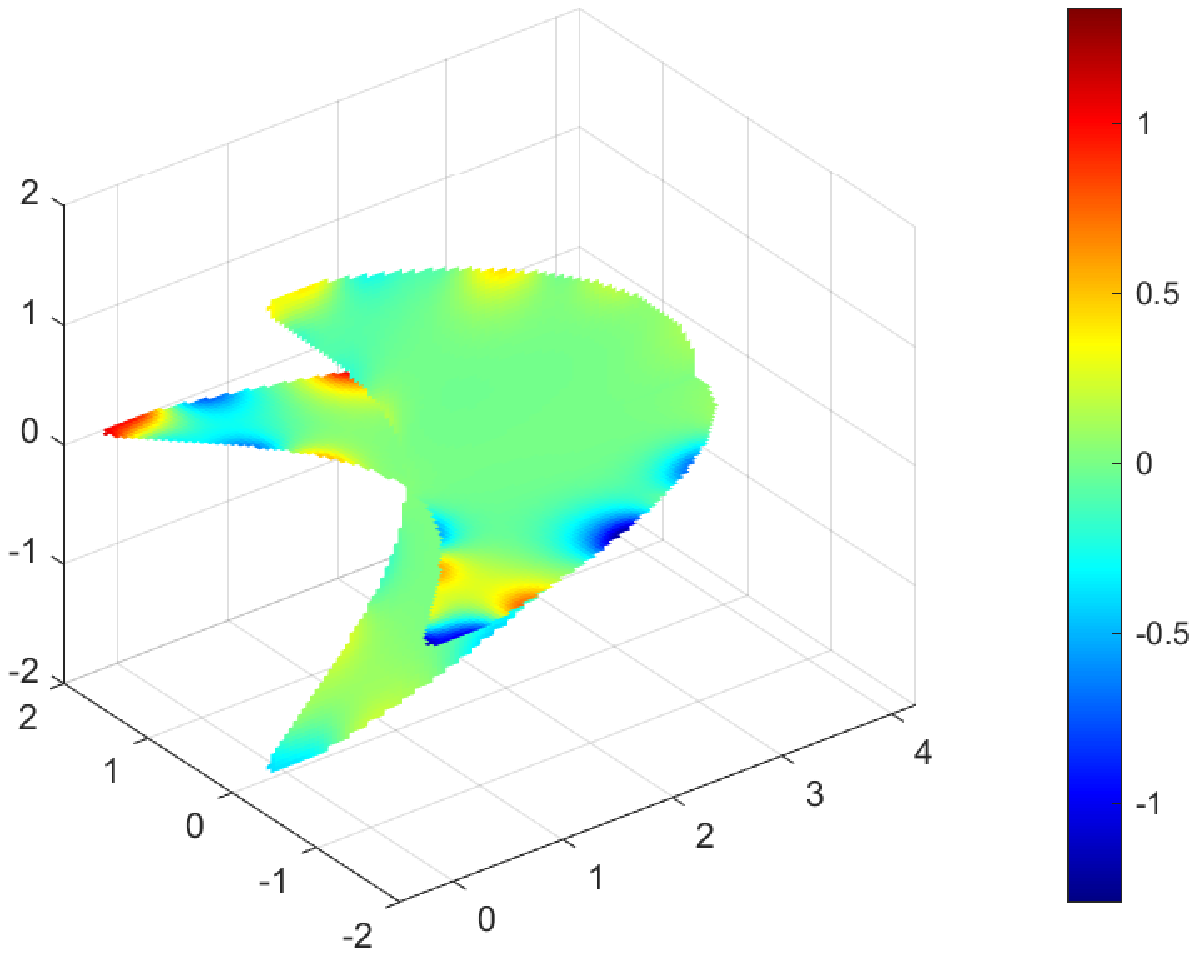}
	}
	\caption{(a) \& (b): $v$ to \eqref{eq:syst2} associated with $\rho_e=10$, $\kappa=\rho_b=\tilde\lambda=\tilde\mu=1$ and $w^2=15.0140, 5.9875$, respectively. }
	\label{fig:2}
\end{figure}



\section*{Acknowledgment}
The work of H. Diao is supported by a startup fund from National Key R\&D Program of China (No. 2020YFA0714102) and NSFC/RGC Joint Research Grant No. 12161160314. The work of H. Li was supported by  Direct Grant for Research, CUHK (project 4053518).  The work of H. Liu is supported by the Hong Kong RGC General Research Funds (projects 11311122, 11300821 and 12301420),  the NSFC/RGC Joint Research Fund (project N\_CityU101/21), and the ANR/RGC Joint Research Grant, A\_CityU203/19.

\end{document}